\documentclass{amsart}

\usepackage[vcentermath]{youngtab}
\usepackage{amssymb}

\newtheorem{theorem}{Theorem}
\newtheorem{lemma}{Lemma}
\newtheorem{proposition}{Proposition}

\theoremstyle{definition}

\newtheorem{example}{Example}
\newtheorem{remark}{Remark}
\newtheorem{notation}{Notation}

\begin{document}

\title[On smooth orbital varieties]{On the existence of smooth orbital varieties in simple Lie algebras}

\author{Lucas Fresse}\thanks{L. Fresse is supported in part by the ISF Grant Nr. 797/14 and by the ANR project GeoLie ANR-15-CE40-0012.}
\address{Institut \'Elie Cartan, Universit\'e de Lorraine, 54506 Vandoeu\-vre-l\`es-Nancy, France}
\email{lucas.fresse@univ-lorraine.fr}
\author{Anna Melnikov}
\address{Department of Mathematics,
University of Haifa, Haifa 31905, Israel}
\email{melnikov@math.haifa.ac.il}

\date{\today}

\maketitle

\begin{abstract}
The orbital varieties are the irreducible components of the
intersection between a nilpotent orbit and a Borel
subalgebra of the Lie algebra of a reductive group. There is a
geometric correspondence between orbital varieties and irreducible components of
Springer fibers. In type A, a construction due to Richardson implies
that every nilpotent orbit contains at least one smooth orbital
variety and every Springer fiber contains at least one smooth
component. In this paper, we show that this property is also true
for the other classical cases. Our proof uses the interpretation of
Springer fibers as varieties of isotropic flags and van Leeuwen's
parametrization of their components in terms of domino tableaux. In
the exceptional cases, smooth orbital varieties do not arise in
every nilpotent orbit, as already noted by Spaltenstein. We however
give a (non-exhaustive) list of nilpotent orbits which have this
property. Our treatment of exceptional cases relies on an induction
procedure for orbital varieties, similar to the induction procedure
for nilpotent orbits.
\end{abstract}

\section{Introduction}

\subsection{Springer fibers, and a question}

\label{section-1.1}

Let $G$ be a connected reductive algebraic group over $\mathbb{K}$
(an algebraically closed field of characteristic zero), of Lie
algebra $\mathfrak{g}$. The flag variety of $G$ is defined by
$\mathcal{B}=G/B$, where $B\subset G$ is a Borel subgroup. It can
also be viewed as the set of Borel subalgebras
$\mathfrak{b}\subset\mathfrak{g}$.

For every nilpotent element $x\in\mathfrak{g}$, we define
\[\mathcal{B}_x=\{\mathfrak{b}\in\mathcal{B}:x\in\mathfrak{b}\}.\]
Then $\mathcal{B}_x$ is a projective subvariety of $\mathcal{B}$. The varieties $\mathcal{B}_x$ are called Springer fibers. %They are important objects arising in geometric representation theory.

Every variety $\mathcal{B}_x$ is connected and equidimensional \cite{Spaltenstein.Topology}.
Its irreducible components may be singular \cite[\S II.11]{Spaltenstein} and they have an intricate intersection pattern.
The geometry of Springer fibers is a challenging topic of current interest in geometric representation theory (see, e.g., the survey paper \cite{Tymoczko}).

Let us consider the following question:
\begin{equation}
\label{1}
\mbox{does every Springer fiber $\mathcal{B}_x$ have a smooth irreducible component?}
\end{equation}

%Until now the specific geometric properties of the Springer fiber components have been mainly studied in the case where $G=\mathrm{SL}_n(\mathbb{K})$.
The answer to (\ref{1}) is known to be positive in the case of
$G=\mathrm{SL}_n(\mathbb{K})$ \cite[\S II.5]{Spaltenstein} (see also
Section \ref{section-SLn} below). For some particular nilpotent
elements $x\in\mathfrak{sl}_n(\mathbb{K})$ we even have that all the
components of $\mathcal{B}_x$ are smooth \cite{FM-2010}.
%Furthermore there are special families of nilpotent elements $x\in\mathfrak{sl}_n(\mathbb{K})$ for which all the components of $\mathcal{B}_x$ are smooth \cite{FM-2010}.
When $x$ runs over all the nilpotent classes of $\mathfrak{sl}_n(\mathbb{K})$, most of the components of the $\mathcal{B}_x$'s nevertheless appear to be singular \cite[Appendix 2]{FM-2013}.
%However if we let $x$ run over all the nilpotent classes of $\mathfrak{sl}_n(\mathbb{K})$, then it turns out that most of the components of the $\mathcal{B}_x$'s are singular \cite[Appendix 2]{FM-2013}.

The answer to (\ref{1}) is negative in the case of exceptional
groups, already in type $\mathrm{G}_2$ \cite[\S II.11.4]{Spaltenstein} (see also Proposition \ref{P2} below).

In this paper, {\em we affirmatively answer
%give an affirmative answer to
question (\ref{1}) in the case where $G$ is one of the other classical groups $\mathrm{SO}_n(\mathbb{K})$ and $\mathrm{Sp}_{2n}(\mathbb{K})$}.

Our main result (stated in Theorem \ref{T1} below) is actually more precise.
Before stating this result, we describe a parallel approach to question (\ref{1})
in the following subsections \ref{section-1.2}--\ref{section-1.3}.

\subsection{Orbital varieties, and another question}
\label{section-1.2}
Let $G\times\mathfrak{g}\to\mathfrak{g}$, $(g,x)\mapsto g\cdot x$ denote the adjoint action.
We fix a Borel subgroup $B\subset G$ and denote by $\mathfrak{n}$ the nilpotent radical of its Lie algebra.
%We fix a Borel subgroup $B\subset G$ and denote by $\mathfrak{b}$ and $\mathfrak{n}$ the corresponding Borel subalgebra of $\mathfrak{g}$ and its nilpotent radical.
The set $\mathcal{N}:=G\cdot\mathfrak{n}$ is the nilpotent cone of $\mathfrak{g}$.
It consists of finitely many adjoint ($G$-)orbits, called nilpotent orbits.

Any nilpotent orbit $\mathcal{O}_x:=G\cdot x$ is a quasi-affine
algebraic variety, hence so is the intersection
$\mathcal{O}_x\cap\mathfrak{n}$. The variety
$\mathcal{O}_x\cap\mathfrak{n}$ is equidimensional of dimension
$\frac{1}{2}\dim\mathcal{O}_x$~\cite{Spaltenstein.Topology}. Its
irreducible components are called orbital varieties.
%These objects arise in the study of quantization of highest weight modules of the enveloping algebra $U(\mathfrak{g})$ and in the study of Springer Weyl group representations.
We ask:
\begin{equation}
\label{2}
\mbox{is there a smooth orbital variety in every nilpotent orbit $\mathcal{O}_x$ of $\mathfrak{g}$?}
\end{equation}
In the following subsection we stress that orbital varieties and Springer fiber components,
as well as questions (\ref{1}) and (\ref{2}),
are closely related.

\subsection{Relation between questions (\ref{1}) and (\ref{2})}
\label{section-1.3}
Note that the Borel subgroup $B$ acts on the variety $\mathcal{O}_x\cap\mathfrak{n}$.
The group $B$ is connected, hence every orbital variety of $\mathcal{O}_x\cap\mathfrak{n}$ is $B$-stable.
Also the stabilizer $Z_G(x):=\{g\in G:g\cdot x=x\}$ acts on the Springer fiber $\mathcal{B}_x$.
The group $Z_G(x)$ is not connected in general, hence the finite group $A_x:=Z_G(x)/Z_G(x)^0$ acts by permuting the irreducible components of $\mathcal{B}_x$.

Relying on the two smooth maps $G\to\mathcal{B}=G/B$ and $G\to\mathcal{O}_x=G/Z_G(x)$, we have \cite{Spaltenstein.Topology, FM-2013}:

\begin{proposition}
\label{P1}
Let $x\in\mathfrak{g}$ be nilpotent.
Then, there is a one-to-one correspondence $X\mapsto \Xi(X)$ between the orbital varieties of $\mathcal{O}_x\cap\mathfrak{n}$ and the $A_x$-orbits of irreducible components of  $\mathcal{B}_x$. This correspondence satisfies the following properties:
\begin{itemize}
\item[\rm (a)] The orbital variety $X$ is smooth if and only if the components $C\in\Xi(X)$ are smooth and pairwise disjoint;
\item[\rm (b)] In particular if $\mathcal{O}_x\cap\mathfrak{n}$ contains a smooth orbital variety, then $\mathcal{B}_x$ contains a smooth component; and if $\mathcal{B}_x$ contains a smooth, $Z_G(x)$-stable component, then $\mathcal{O}_x\cap\mathfrak{n}$ contains a smooth orbital variety.
\end{itemize}
\end{proposition}

By Proposition \ref{P1}\,{\rm (b)} a positive answer to (\ref{2}) yields a positive answer to (\ref{1}).

\begin{remark}
The equivalence in Proposition \ref{P1}\,{\rm (a)} cannot be
improved: in Section \ref{section-6} we point out an example of a
singular orbital variety $X$ in $\mathfrak{sp}_6(\mathbb{K})$ such that the
corresponding Springer fiber components $C\in\Xi(X)$ are smooth (but
not pairwise disjoint).
\end{remark}

\subsection{Main result in classical cases}

In this paper we give an affirmative answer to (\ref{1}) and (\ref{2}) in the classical cases.

\begin{theorem}
\label{T1}
Let $G$ be one of the classical groups $\mathrm{SL}_n(\mathbb{K})$, $\mathrm{SO}_n(\mathbb{K})$, $\mathrm{Sp}_{2n}(\mathbb{K})$.
Then:
\begin{itemize}
\item[\rm (a)] Every nilpotent orbit $\mathcal{O}\subset\mathfrak{g}$ contains at least one smooth orbital variety.
\item[\rm (b)] In fact, for every nilpotent element $x\in\mathfrak{g}$, the Springer fiber $\mathcal{B}_x$ has at least one smooth, $Z_G(x)$-stable irreducible component.
\end{itemize}
\end{theorem}

The proof is given in Section \ref{section-5.3}.

\subsection{Partial results in exceptional cases}

We complete our main result by partially answering questions
(\ref{1}) and (\ref{2}) in exceptional cases:

\begin{proposition}
\label{P2} Let $G$ be a simple algebraic group and let
$\mathcal{O}\subset\mathfrak{g}$ be a nilpotent orbit. Assume that
one of the following conditions occurs.
\begin{itemize}
\item[\rm (a)] $G$ is of type $\mathrm{G}_2$ and $\mathcal{O}$ is not $A_1$;
\item[\rm (b)] $G$ is of type $\mathrm{F}_4$ and $\mathcal{O}$ is not $A_1$, $\widetilde{A}_1$, $A_1+\widetilde{A}_1$, $A_2+\widetilde{A}_1$, $\widetilde{A}_2+A_1$ (i.e., not rigid);
\item[\rm (c)] $G$ is of type $\mathrm{E}_6$ and $\mathcal{O}$ is not $A_1$, $3A_1$, $2A_2+A_1$ (i.e., not rigid);
\item[\rm (d)] $G$ is of type $\mathrm{E}_7$ and $\mathcal{O}$ is not $A_1$, $2A_1$, $(3A_1)'$, $4A_1$, $A_2+2A_1$, $A_1+2A_2$, $(A_1+A_3)'$ (i.e., not rigid) and neither $A_2+A_1$, $A_3+2A_1$, $A_5+A_1$ (so that $\mathcal{O}$ is induced by an orbit of a Levi subgroup of classical type
or by the trivial orbit of a Levi subgroup of type $\mathrm{E}_6$);
\item[\rm (e)] $G$ is of type $\mathrm{E}_8$ and $\mathcal{O}$ is not $A_1$, $2A_1$, $3A_1$, $4A_1$, $A_2+A_1$, $A_2+2A_1$, $A_2+3A_1$, $2A_2+A_1$, $A_3+A_1$, $2A_2+2A_1$, $A_3+2A_1$, $D_4(a_1)+A_1$, $A_3+A_2+A_1$, $2A_3$, $D_5(a_1)+A_2$, $A_5+A_1$, $A_4+A_3$ (i.e., not rigid) and neither $A_3$, $D_4+A_1$, $A_4+A_1$, $D_5(a_1)$, $D_5(a_1)+A_1$, $E_6(a_3)+A_1$, $E_7(a_5)$, $D_5+A_1$, $E_6+A_1$ (so that $\mathcal{O}$ is induced by an orbit of a Levi subgroup of classical type
or by the trivial orbit of a Levi subgroup of type $\mathrm{E}_6$ or $\mathrm{E}_7$).
\end{itemize}
Then the nilpotent orbit $\mathcal{O}$ contains at least one smooth orbital variety.

If $G$ is of type $\mathrm{G}_2$ and $\mathcal{O}=\mathcal{O}_x$ is
the nilpotent orbit $A_1$, then $\mathcal{O}_x\cap\mathfrak{n}$ and
$\mathcal{B}_x$ are both irreducible and singular, hence there is no
smooth orbital variety in $\mathcal{O}_x$ and no smooth component in
$\mathcal{B}_x$.
\end{proposition}

The proof is done in Section \ref{section-2}.

\subsection{Organization of the paper}

The proof of Proposition \ref{P2} is based on the following inductive
principle:

\smallskip

{\it whenever $\mathcal{O}_L$ is a nilpotent orbit of a Levi subgroup $L\subset G$ which contains a smooth orbital variety, the induced nilpotent orbit $\mathrm{Ind}_L^G(\mathcal{O}_L)$ contains a smooth orbital variety}

\smallskip
\noindent
%\begin{center}
%{\it whenever $\mathcal{O}_L$ is a nilpotent orbit of a Levi subgroup $L\subset G$ which contains a %smooth orbital variety, the induced nilpotent orbit $\mathrm{Ind}_L^G(\mathcal{O}_L)$ contains a smooth %orbital variety.}
%\end{center}
(shown in Section \ref{section-2}) and on the description of induced
nilpotent orbits in exceptional simple Lie algebras given in
\cite{DeGraaf-Elashvili}.
In particular, the nilpotent orbits listed in Proposition \ref{P2}\,{\rm (b)}--{\rm (e)} being induced from nilpotent orbits of classical type,
parts {\rm (b)}--{\rm (e)} of Proposition \ref{P2} are implied by Theorem \ref{T1} and the above inductive principle. The claims in Proposition \ref{P2} which concern type $\mathrm{G}_2$ are shown in Section \ref{section-2.2}.
Before that, in Section \ref{section-2.1} we review the induction procedure for nilpotent orbits
and in Section \ref{section-3.2} we define a similar induction procedure for orbital varieties.
The above inductive principle is obtained as a consequence of this construction
(which can also be of independant interest).

Our approach to classical cases is different. In the situation of
Theorem \ref{T1}, the nilpotent elements $x\in\mathfrak{g}$ are
nilpotent endomorphisms of a standard representation $V$ of $G$, and
the Springer fibers $\mathcal{B}_x$ can be viewed as varieties of
$x$-stable complete flags (i.e., maximal chains of $x$-stable
subspaces of $V$). For $G=\mathrm{SL}_n(\mathbb{K})$ (type A), this
description of $\mathcal{B}_x$ is explained in Section
\ref{section-SLn}. In this case, Theorem \ref{T1} is a well-known
fact, whose proof is also recalled in Section \ref{section-SLn}.

For $G=\mathrm{SO}_n(\mathbb{K})$ or $\mathrm{Sp}_{2n}(\mathbb{K})$
(types B, C, D), the nilpotent endomorphism $x$ is skew-adjoint with
respect to a given orthogonal or symplectic form $\omega$ on $V$,
and $\mathcal{B}_x$ can be viewed as a variety of complete flags
which are $x$-stable and isotropic, that is, fixed by the involution
$\sigma$ which maps a complete flag $(V_0,\ldots,V_n)$ to the flag
$(V_n^\perp,\ldots,V_0^\perp)$ of orthogonal subspaces with respect
to $\omega$. In particular, $\mathcal{B}_x$ corresponds to the
subvariety of $\sigma$-fixed points of the type A Springer fiber
attached to $x$ and each irreducible component of $\mathcal{B}_x$
lies in the fixed-point set of a component of the type A
Springer fiber. These facts are explained in Section
\ref{section-4}. In Section \ref{section-4.4}, we also recall the
classification of the components of $\mathcal{B}_x$ in terms of
admissible domino tableaux, proved in \cite{van-Leeuwen}. The
results in \cite{van-Leeuwen} also include a combinatorial
description of the action of the stabilizer $Z_G(x)$ on the set of
components of $\mathcal{B}_x$, which we need for the proof of
Theorem \ref{T1}.

The proof of Theorem \ref{T1} is given in Section \ref{section-5}.
It is based on the construction of a suitable family of smooth
components of type A Springer fibers, parameterized by well-chosen
domino tableaux. The $\sigma$-fixed point set of each of these type
A components contains a single Springer fiber component for the
considered classical type (B, C, or D), whose smoothness is shown by
invoking the following general fact:

\begin{proposition}[\cite{Fogarty, Iversen}]
\label{P3}
Let $X$ be an algebraic variety over $\mathbb{K}$. Let $H$ be a
linearly reductive group with algebraic action on $X$ (for instance
a finite group). If the variety $X$ is smooth, then the fixed point
set $X^H:=\{x\in X:\forall h\in H,\ h(x)=x\}$ is also smooth.
\end{proposition}

Finally, in Section \ref{section-6}, we give an example of nilpotent
element $x\in\mathfrak{sp}_6(\mathbb{K})$ such that all irreducible
components of $\mathcal{B}_x$ are smooth, but some orbital varieties
of $\mathcal{O}_x$ are singular. This fact, which underlines the
subtle character of the correspondence between orbital varieties and
Springer fiber components described in Proposition \ref{P1}, is
caused by the nontriviality of the action of $A_x(=Z_G(x)/Z_G(x)^0$)
on the set of components of $\mathcal{B}_x$ and by the complexity of
the intersection pattern of these components.

\section{Induction procedure}

\label{section-2}

In this section, unless otherwise stated, $G$ is a connected
reductive group.

\subsection{Review on induction of nilpotent orbits}

\label{section-2.1}

Let $P\subset G$ be a parabolic subgroup, with unipotent radical
$U_P$ and a Levi factor $L$. Let $\mathfrak{p}$, $\mathfrak{n}_P$,
and $\mathfrak{l}$ be the Lie algebras of $P$, $U_P$, and $L$; thus
$\mathfrak{p}=\mathfrak{l}\oplus\mathfrak{n}_P$.

\begin{proposition}[\protect{\cite{Lusztig-Spaltenstein}}]
\label{P4} Let $\mathcal{O}_L\subset\mathfrak{l}$ be a nilpotent
($L$-)orbit. The set
$\mathcal{O}_L+\mathfrak{n}_P:=\{x+y:x\in\mathcal{O}_L,\
y\in\mathfrak{n}_P\}$ is an irreducible, locally closed subset of
$\mathfrak{g}$ consisting of nilpotent elements, so that there is a
unique nilpotent $G$-orbit
$\mathcal{O}=:\mathrm{Ind}_L^G(\mathcal{O}_L)$ such that
$\mathcal{O}\cap(\mathcal{O}_L+\mathfrak{n}_P)$ is open and dense in
$\mathcal{O}_L+\mathfrak{n}_P$. Moreover,
\begin{itemize}
\item[\rm (a)] Every nilpotent $G$-orbit
$\mathcal{O}'$ such that $\mathcal{O}'\cap (\mathcal{O}_L+\mathfrak{n}_P)\not=\emptyset$ satisfies $\mathcal{O}'\subset\overline{\mathrm{Ind}_L^G(\mathcal{O}_L)}$;
\item[\rm (b)] $\dim \mathrm{Ind}_L^G(\mathcal{O}_L)=\dim
\mathcal{O}_L+2\dim\mathfrak{n}_P$;
\item[\rm (c)] $\mathrm{Ind}_L^G(\mathcal{O}_L)\cap(\mathcal{O}_L+\mathfrak{n}_P)$
is $P$-stable and consists of a single $P$-orbit.
\end{itemize}
\end{proposition}

A nilpotent $G$-orbit is said to be induced if is of the form
$\mathrm{Ind}_L^G(\mathcal{O}_L)$ for some proper Levi factor
$L\subset G$, and it is said to be rigid
otherwise. We refer to \cite{Collingwood-McGovern} and
\cite{DeGraaf-Elashvili} for an explicit description of
induced/rigid orbits in the classical and exceptional cases.

\subsection{Induction of orbital varieties}

\label{section-3.2}

As in Section \ref{section-1.2} we fix a Borel subgroup $B\subset G$
and let $\mathfrak{n}\subset\mathfrak{g}$ be the nilradical of its
Lie algebra. We take $P,L,U_P$ and
$\mathfrak{p},\mathfrak{l},\mathfrak{n}_P$ as in Section
\ref{section-2.1}. Up to replacing $P,L$ by conjugates we may assume
all these data compatible in the sense that
\begin{eqnarray*}
 & \mbox{$B\subset P$, so that $U_P\subset B$ and
$\mathfrak{n}_P\subset\mathfrak{n}$,} \\
 & \mbox{$B_L:=B\cap L$ is a Borel subgroup of $L$, so that $B=B_LU_P$
and $\mathfrak{n}=\mathfrak{n}_L\oplus\mathfrak{n}_P$,}
\end{eqnarray*}
where $\mathfrak{n}_L:=\mathfrak{n}\cap\mathfrak{l}$ is the
nilradical of the Lie algebra of $B_L$.

\begin{proposition}
\label{P5} Let $\mathcal{O}_L\subset\mathfrak{l}$ be a nilpotent
$L$-orbit and let $\mathcal{O}:=\mathrm{Ind}_L^G(\mathcal{O}_L)$ be
the corresponding induced $G$-orbit. Let $X_L$ be an orbital variety
of $\mathcal{O}_L$, i.e., an irreducible component of
$\mathcal{O}_L\cap\mathfrak{n}_L$. Then
$X:=\mathcal{O}\cap(X_L+\mathfrak{n}_P)$ is an orbital variety of
$\mathcal{O}$, i.e., an irreducible component of
$\mathcal{O}\cap\mathfrak{n}$. Moreover, if $X_L$ is smooth, then
$X$ is smooth.
\end{proposition}

\begin{proof}
Since $X_L+\mathfrak{n}_P\subset\mathfrak{n}$, we know that $X$ is a
subset of $\mathcal{O}\cap\mathfrak{n}$. We first note that $X$ is
nonempty. Indeed, by definition of $\mathcal{O}$, we have
$\mathcal{O}\cap(\mathcal{O}_L+\mathfrak{n}_P)\not=\emptyset$, thus
there are $x\in\mathcal{O}_L$ and $x'\in\mathfrak{n}_P$ such that
$x+x'\in \mathcal{O}$. Since $X_L\subset\mathcal{O}_L$, we can find
$\ell\in L$ such that $\ell\cdot x\in X_L$. Since $\mathfrak{n}_P$
and $\mathcal{O}$ are $L$-stable, we conclude that $\ell\cdot(x+x')$
is an element of $\mathcal{O}\cap(X_L+\mathfrak{n}_P)=X$.

The map $\mathfrak{n}_L\times\mathfrak{n}_P\to\mathfrak{n}$,
$(x,x')\mapsto x+x'$ is an isomorphism. This guarantees that
$X_L+\mathfrak{n}_P\cong X_L\times\mathfrak{n}_P$ is an irreducible
closed subset of $\mathfrak{n}$ whose dimension is
\[\dim (X_L+\mathfrak{n}_P)=\dim X_L+\dim \mathfrak{n}_P=\frac{1}{2}(\dim\mathcal{O}_L+2\dim\mathfrak{n}_P)=\frac{1}{2}\dim\mathcal{O},\]
where we use that $\dim X_L=\frac{1}{2}\dim\mathcal{O}_L$ (see
Section \ref{section-1.2}); moreover $X_L+\mathfrak{n}_P$ is smooth
whenever $X_L$ is smooth.

Writing
$X=\mathcal{O}\cap(\mathcal{O}_L+\mathfrak{n}_P)\cap(X_L+\mathfrak{n}_P)$
and using that $\mathcal{O}\cap(\mathcal{O}_L+\mathfrak{n}_P)$ is
open in $\mathcal{O}_L+\mathfrak{n}_P$, we deduce that $X$ is open
in $X_L+\mathfrak{n}_P$. It follows that $X$ is irreducible and
$\dim X=\frac{1}{2}\dim\mathcal{O}=\dim\mathcal{O}\cap\mathfrak{n}$
(see Section \ref{section-1.2}); moreover $X$ is smooth whenever
$X_L$ is smooth. For completing the proof of the proposition (i.e.,
for concluding that $X$ is an orbital variety, that is, an
irreducible component of $\mathcal{O}\cap\mathfrak{n}$), it remains
to check that $X$ is a closed subset of
$\mathcal{O}\cap\mathfrak{n}$. To this end, it suffices to check
that $X$ is closed in $\mathcal{O}$ (since the inclusion
$X\subset\mathcal{O}\cap\mathfrak{n}$ is already known).

The fact that $X_L$ is a closed subset of $\mathcal{O}_L$ implies
that $X$ is a closed subset of
$\mathcal{O}\cap(\mathcal{O}_L+\mathfrak{n}_P)$. For every $L$-orbit
$\mathcal{O}'_L$ such that
$\mathcal{O}'_L\subset\overline{\mathcal{O}_L}\setminus\mathcal{O}_L$,
the inequality
\[\dim\mathcal{O}=\dim\mathcal{O}_L+2\mathfrak{n}_P>\dim\mathcal{O}'_L+2\mathfrak{n}_P=\dim\mathrm{Ind}_L^G(\mathcal{O}'_L)\]
holds, which guarantees (thanks to Proposition \ref{P4}\,{\rm (a)})
that $\mathcal{O}\cap(\mathcal{O}'_L+\mathfrak{n}_P)=\emptyset$.
Whence
$\mathcal{O}\cap(\mathcal{O}_L+\mathfrak{n}_P)=\mathcal{O}\cap(\overline{\mathcal{O}_L}+\mathfrak{n}_P)$,
so that $X$ is in fact a closed subset of
$\mathcal{O}\cap(\overline{\mathcal{O}_L}+\mathfrak{n}_P)$, and thus
a closed subset of $\mathcal{O}$.
\end{proof}

\begin{remark}
\label{R2} This statement includes as a special case the
construction of the so-called Richardson orbital varieties. The
Richardson nilpotent orbit attached to the parabolic subgroup $P$ is
by definition the unique nilpotent orbit $\mathcal{O}$ which
intersects the nilradical $\mathfrak{n}_P$ along an open dense
subset; in other words $\mathcal{O}$ is induced by the trivial
nilpotent $L$-orbit $\{0\}\subset\mathfrak{l}$. Proposition \ref{P5}
then shows that $\mathcal{O}\cap\mathfrak{n}_P$ is a smooth orbital
variety of $\mathcal{O}$ (called Richardson orbital variety). Every
Richardson nilpotent orbit therefore contains at least one smooth
orbital variety. However, outside of the case of $\mathfrak{g}=\mathfrak{sl}_n(\mathbb{K})$, most of nilpotent orbits are not Richardson (see
\cite{Collingwood-McGovern} or, e.g., \cite[Appendix C]{FM-2013}).
See however Remark \ref{R4}.
\end{remark}

\subsection{Proof of Proposition \ref{P2}}

\label{section-2.2}

The proposition, except part {\rm (a)} and the last sentence,
follows from Proposition \ref{P5}, Theorem \ref{T1}, and the
description of induced nilpotent orbits given in
\cite{DeGraaf-Elashvili}.

The last statement of the proposition appears in \cite[\S
II.11.4]{Spaltenstein} without proof. For the sake of completeness
we include here a proof of this statement.

Assume $G$ simple of type $\mathrm{G}_2$. The Lie algebra
$\mathfrak{g}$ contains five nilpotent orbits of respective
dimensions 0, 6, 8, 10, and 12. We focus on the nilpotent orbits
$\mathcal{O}$ of type $A_1$, which is the minimal nilpotent orbit of
$\mathfrak{g}$, i.e., $\dim\mathcal{O}=6$, and
$\widetilde{\mathcal{O}}$ of type $\widetilde{A}_1$, such that
$\dim\widetilde{\mathcal{O}}=8$. The other two nontrivial nilpotent
orbits are of Richardson type, hence they contain smooth orbital
varieties. We need to show that $\mathcal{O}\cap\mathfrak{n}$ is
irreducible and singular and to construct a smooth orbital variety
in $\widetilde{\mathcal{O}}$.

We fix a maximal torus $T\subset B$, denote by $\mathfrak{t}$ its
Lie algebra, and consider the corresponding root system $\Phi$ and
root space decompositions
$\mathfrak{g}=\mathfrak{t}\oplus\bigoplus_{\gamma\in\Phi}\mathfrak{g}_\gamma$
and $\mathfrak{n}=\bigoplus_{\gamma\in\Phi^+}\mathfrak{g}_\gamma$
where $\Phi^+$ is a system of positive roots determined by two
simple roots $\alpha$ (short) and $\beta$ (long) by
\[
\Phi^+=\{\alpha,\beta,\alpha+\beta,2\alpha+\beta,3\alpha+\beta,3\alpha+2\beta\}.
\]
We denote the positive roots by $\alpha_1,\ldots,\alpha_6$ in the
order they appear in the above description of $\Phi^+$, and we
set $\alpha_{-i}=-\alpha_i$. Let
$\{\lambda_\alpha,\lambda_\beta\}\subset\mathfrak{t}$ be the dual
basis of $\{\alpha,\beta\}\subset\mathfrak{t}^*$. The Lie algebra
$\mathfrak{g}$ has a basis $\{\lambda_\alpha,\lambda_\beta,e_\gamma\
(\gamma\in\Phi)\}$ such that
\[e_\gamma\in\mathfrak{g}_\gamma\quad\mbox{and}\quad
[e_\gamma,e_\delta]=\left\{\begin{array}{ll} 0 & \mbox{if
$\gamma+\delta\notin\Phi\cup\{0\}$,} \\
N_{\gamma,\delta}e_{\gamma+\delta} & \mbox{if
$\gamma+\delta\in\Phi$,} \\
h_\gamma:=\langle\gamma,\alpha\rangle\lambda_\alpha+\langle\gamma,\beta\rangle\lambda_\beta
& \mbox{if $\gamma+\delta=0$, $\gamma\in\Phi^+$,}
\end{array}\right.\]
where $(\langle\gamma,\alpha\rangle,\langle\gamma,\beta\rangle)$ are
the pairs of integers listed in the next table
\[
\begin{array}{|r||c|c|c|c|c|c|c|c|c|}
\hline
i & 1 & 2 & 3 & 4 & 5 & 6 \\
\hline (\langle\alpha_i,\alpha\rangle,\langle\alpha_i,\beta\rangle)
& (2,-3) & (-1,2) & (-1,3) & (1,0) & (1,-1) & (0,1) \\ \hline
\end{array}
\]
and $N_{\gamma,\delta}$ are the coefficients of the matrix
$N=\left(N_{\alpha_i,\alpha_j}\right)_{i,j\in\{-6,\ldots,-1,1,\ldots,6\}}$
given by
\[
N=\mbox{\small$\left(
\begin{array}{cccccccccccc}
0 & 0 & 0 & 0 & 0 & 0 & 0 & 1 & 1 & -1 & -1 & \mbox{-} \\
0 & 0 & 0 & 0 & 1 & 0 & 1 & 0 & 0 & -1 & \mbox{-} & -1 \\
0 & 0 & 0 & 3 & 0 & 3 & -2 & 0 & 2 & \mbox{-} & -1 & -1 \\
0 & 0 & -3 & 0 & 0 & -2 & -3 & 1 & \mbox{-} & 2 & 0 & 1 \\
0 & -1 & 0 & 0 & 0 & -1 & 0 & \mbox{-} & 1 & 0 & 0 & 1 \\
0 & 0 & -3 & 2 & 1 & 0 & \mbox{-} & 0 & -3 & -2 & 1 & 0 \\
0 & -1 & 2 & 3 & 0 & \mbox{-} & 0 & -1 & -2 & 3 & 0 & 0 \\
-1 & 0 & 0 & -1 & \mbox{-} & 0 & 1 & 0 & 0 & 0 & 1 & 0 \\
-1 & 0 & -2 & \mbox{-} & -1 & 3 & 2 & 0 & 0 & 3 & 0 & 0 \\
1 & 1 & \mbox{-} & -2 & 0 & 2 & -3 & 0 & -3 & 0 & 0 & 0 \\
1 & \mbox{-} & 1 & 0 & 0 & -1 & 0 & -1 & 0 & 0 & 0 & 0 \\
\mbox{-} & 1 & 1 & -1 & -1 & 0 & 0 & 0 & 0 & 0 & 0 & 0
\end{array}
\right)$}
\]
(see \cite{Carter, Mayanskiy}). Consider an arbitrary element
\[x=\sum_{i=1}^6x_ie_{\alpha_i}\in\mathfrak{n}.\]
The matrix $C_x$ of $[x,\cdot]:\mathfrak{g}\to\mathfrak{g}$ in the
basis
$(e_{-\alpha_6},\ldots,e_{-\alpha_1},\lambda_\alpha,\lambda_\beta,e_{\alpha_1},\ldots,e_{\alpha_6})$
is
\[
\mbox{\small\setlength{\arraycolsep}{3pt}
$\left(\begin{array}{cccccccccccccc}
0 & 0 & 0 & 0 & 0 & 0 & 0 & 0 & 0 & 0 & 0 & 0 & 0 & 0 \\
-x_2 & 0 & 0 & 0 & 0 & 0 & 0 & 0 & 0 & 0 & 0 & 0 & 0 & 0 \\
-x_3 & -x_1 & 0 & 0 & 0 & 0 & 0 & 0 & 0 & 0 & 0 & 0 & 0 & 0 \\
x_4 & 0 & 2x_1 & 0 & 0 & 0 & 0 & 0 & 0 & 0 & 0 & 0 & 0 & 0 \\
x_5 & 0 & 0 & 3x_1 & 0 & 0 & 0 & 0 & 0 & 0 & 0 & 0 & 0 & 0 \\
0 & x_4 & -2x_3 & -x_2 & 0 & 0 & 0 & 0 & 0 & 0 & 0 & 0 & 0 & 0 \\
0 & x_5 & x_4 & -x_3 & -x_2 & 2x_1 & 0 & 0 & 0 & 0 & 0 & 0 & 0 & 0 \\
x_6 & -x_5 & 0 & 3x_3 & 2x_2 & -3x_1 & 0 & 0 & 0 & 0 & 0 & 0 & 0 & 0 \\
0 & 0 & x_5 & -2x_4 & -x_3 & 0 & -x_1 & 0 & 0 & 0 & 0 & 0 & 0 & 0 \\
0 & x_6 & 0 & 0 & 0 & 3x_3 & 0 & -x_2 & 0 & 0 & 0 & 0 & 0 & 0 \\
0 & 0 & x_6 & 0 & 0 & 2x_4 & -x_3 & -x_3 & x_2 & -x_1 & 0 & 0 & 0 & 0 \\
0 & 0 & 0 & -x_6 & 0 & -x_5 & -2x_4 & -x_4 & 2x_3 & 0 & -2x_1 & 0 & 0 & 0 \\
0 & 0 & 0 & 0 & -x_6 & 0 & -3x_5 & -x_5 & -3x_4 & 0 & 0 & 3x_1 & 0 & 0 \\
0 & 0 & 0 & 0 & 0 & 0 & -3x_6 & -2x_6 & 0 & -x_5 & -3x_4 & 3x_3 &
x_2 & 0
\end{array}\right).$}
\]
Note that
\[
\dim G\cdot x=\dim G-\dim Z_G(x)=\dim\mathfrak{g}-\dim\ker
[x,\cdot]=\mathrm{rank}\,C_x.
\]

Since $\mathcal{O}$ is the unique nilpotent orbit of $\mathfrak{g}$
of dimension $6$, we obtain the characterization
\[x\in\mathcal{O}\cap\mathfrak{n}\Leftrightarrow \dim G\cdot x=6\Leftrightarrow \mathrm{rank}\,C_x=6.\]
Note that the form of the matrix $C_x$ easily implies
\begin{eqnarray*}
 \mathrm{rank}\,C_x=6 & \Rightarrow & x_1=0,
 \\
 \mathrm{rank}\,C_x=6\ \mbox{ and }\ x_2=0& \Leftrightarrow & x_1=x_2=x_3=x_4=0\ \mbox{ and }\ (x_5,x_6)\not=(0,0),
\end{eqnarray*}
and for $x\in\mathfrak{n}$ such that $x_1=0$, $x_2\not=0$, by means
of elementary operations on the matrix $C_x$, we can see that
\begin{eqnarray*}
\mathrm{rank}\,C_x=6 & \Leftrightarrow & x_2x_4=x_3^2,\ \
x_3x_5=-x_4^2,\ \mbox{ and }\ x_2x_5=-x_3x_4.
\end{eqnarray*}
Finally we obtain
\[\mathcal{O}\cap\mathfrak{n}=\{x\in\mathfrak{n}\setminus\{0\}:x_1=x_2x_4-x_3^2=x_3x_5+x_4^2=x_2x_5+x_3x_4=0\}.\]
It is easy to deduce that $\mathcal{O}\cap\mathfrak{n}$ is
irreducible and singular at every point $x$ such that
$x_1=\ldots=x_5=0$ and $x_6\not=0$. Since $Z_G(x)$ is connected
whenever $x\in\mathcal{O}$ (see \cite[\S
8.4]{Collingwood-McGovern}), we conclude from Proposition \ref{P1}
that the Springer fiber $\mathcal{B}_x$ is also irreducible and
singular. This shows the last statement in Proposition \ref{P2}.

Let
\[\mathcal{V}:=\{x\in\mathfrak{n}:x_2=3x_4^2+4x_3x_5-4x_1x_6=0\},\]
which is a 4-dimensional, irreducible, closed subvariety of
$\mathfrak{n}$ whose sole singular point is $0$. For every
$x\in\mathcal{V}$ we can see that
\[\mathrm{rank}\,C_x\leq 8\qquad\mbox{and}\qquad \mathrm{rank}\,C_x=8\Leftrightarrow(x_1,x_3)\not=(0,0).\]
Hence
$\widetilde{\mathcal{V}}:=\mathcal{V}\cap\widetilde{\mathcal{O}}=\{x\in\mathcal{V}:(x_1,x_3)\not=(0,0)\}$
is a smooth, irreducible, closed subvariety of
$\widetilde{\mathcal{O}}\cap\mathfrak{n}$ such that
$\dim\widetilde{\mathcal{V}}=4=\dim\widetilde{\mathcal{O}}\cap\mathfrak{n}$.
Therefore $\widetilde{\mathcal{V}}$ is a smooth irreducible
component of $\widetilde{\mathcal{O}}\cap\mathfrak{n}$, i.e., a
smooth orbital variety of $\widetilde{\mathcal{O}}$. This completes
the proof of the proposition.

\begin{remark}
Assume in this remark that $G$ is simple. Fix a root system $\Phi$,
a root space decomposition
$\mathfrak{g}=\mathfrak{t}\oplus\bigoplus_{\gamma\in\Phi}\mathfrak{g}_\gamma$,
root vectors $e_\gamma\in\mathfrak{g}_\gamma$, and a set of simple
roots $\Pi\subset\Phi$, compatible with the choice of $B$ and
$\mathfrak{n}$. Then, letting $\mathcal{O}_\mathrm{min}$ be the
minimal nilpotent orbit of $\mathfrak{g}$, it is shown in \cite[\S
6.2]{Braverman-Joseph} that the map $\gamma\mapsto \overline{B\cdot
e_\gamma}\cap\mathcal{O}_{\mathrm{min}}$ is a one-to-one
correspondence between the long roots $\gamma\in\Pi$ and the orbital
varieties of $\mathcal{O}_\mathrm{min}$, i.e., the irreducible
components of $\mathcal{O}_\mathrm{min}\cap\mathfrak{n}$. This
general property yields the irreducibility of
$\mathcal{O}_\mathrm{min}\cap\mathfrak{n}$ when $G$ is of type
$\mathrm{G}_2$, which is retrieved in the above proof (where it is
also shown that this intersection is singular). When
$G=\mathrm{Sp}_{2n}(\mathbb{K})$ we similarly obtain that
$\mathcal{O}_\mathrm{min}\cap\mathfrak{n}$ is irreducible; however,
this time, it is smooth (see Theorem \ref{T1}).
\end{remark}

\section{Proof of Theorem \ref{T1} in the case $G=\mathrm{SL}_n(\mathbb{K})$}

\label{section-SLn}

In the case of $G=\mathrm{SL}_n(\mathbb{K})$, Theorem \ref{T1} is
already well known (see \cite{Spaltenstein}), but we give here a
proof for the sake of completeness. Part of the notation introduced
in this section is also used in Sections \ref{section-4}--\ref{section-6}.

%\subsection{Flag variety and Springer fibers}

Hereafter in this section we assume that
$G=\mathrm{SL}_n(\mathbb{K})$, so the flag variety $\mathcal{B}$ is
naturally isomorphic to the variety $\mathrm{Fl}(V)$ of complete
flags of the space $V=\mathbb{K}^n$, i.e.,
\[\mathrm{Fl}(V)=\{(V_0\subset V_1\subset\ldots\subset V_n):\mbox{$V_i$ is an $i$-dimensional subspace of $V$ for all $i$}\}.\]
Moreover a nilpotent element $x\in\mathfrak{sl}_n(\mathbb{K})$ is a
nilpotent endomorphism of $V$, and the Springer fiber
$\mathcal{B}_x\subset\mathcal{B}$ can be identified with the
subvariety of $x$-stable complete flags, i.e.,
\[\mathrm{Fl}_x(V):=\{(V_0,\ldots,V_n)\in\mathrm{Fl}(V):x(V_i)\subset V_i\ \mbox{for all $i$}\}.\]

Note that $Z_{\mathrm{SL}_n(\mathbb{K})}(x)\subset
Z_{\mathrm{GL}_n(\mathbb{K})}(x)$. Since the group
$Z_{\mathrm{GL}_n(\mathbb{K})}(x)$ is connected, its action on
$\mathrm{Fl}_x(V)\cong\mathcal{B}_x$ stabilizes each irreducible
component. It follows that every component of $\mathcal{B}_x$ is
$Z_{\mathrm{SL}_n(\mathbb{K})}(x)$-stable. Thus, in view of
Proposition \ref{P1}, questions (\ref{1}) and (\ref{2}) are in fact
equivalent in the case of $\mathrm{SL}_n(\mathbb{K})$, and for
proving Theorem \ref{T1}, it suffices to show the existence of a
smooth component of $\mathcal{B}_x\cong\mathrm{Fl}_x(V)$ for all
$x\in\mathfrak{sl}_n(\mathbb{K})$ nilpotent.

Let $\lambda(x)=(\lambda_1,\ldots,\lambda_k)$ be the list of the
lengths of the Jordan blocks of $x$, written in nonincreasing order.
Thus $\lambda(x)$ is a partition of $n$. Let
$\lambda(x)^*=(\lambda_1^*,\ldots,\lambda_{\lambda_1}^*)$ denote the
dual partition of $n$, that is,
\[\lambda_j^*:=|\{i=1,\ldots,k:\lambda_i\geq j\}|=\dim\ker x^j/\ker
x^{j-1}\ \mbox{ for all $j=1,\ldots,\lambda_1$.}\]
%so
%\[\dim\ker x^j=\lambda_1^*+\ldots+\lambda_j^*\ \mbox{ for all $j=1,\ldots,\lambda_1$.}\]
Many properties of the Springer fiber
$\mathcal{B}_x\cong\mathrm{Fl}_x(V)$ can be described
combinatorially in terms of the partition $\lambda(x)$. For
instance, $\mathrm{Fl}_x(V)$ is equidimensional of dimension
\begin{equation}
\label{dimension-Flx}
\dim\mathrm{Fl}_x(V)=\sum_{j=1}^{\lambda_1}\binom{\lambda_j^*}{2}.
\end{equation}
In addition, the irreducible components of $\mathrm{Fl}_x(V)$ can be
parameterized by the standard Young tableaux of shape $\lambda(x)$.
We refer to \cite[\S II.5]{Spaltenstein} or Section
\ref{section-4.1} below for more details.

Note that the iterated kernels of $x$ form a partial flag of $V$,
\begin{equation}
\label{3} (\ker x\subset \ker x^2\subset \ldots\subset \ker
x^{\lambda_1}=V).
\end{equation}
We consider the subset $\mathcal{K}\subset\mathrm{Fl}(V)$ formed by
complete flags which refine this partial flags, i.e., which include
all the subspaces $\ker x^j$ for $j=1,\ldots,\lambda_1$, in other
words:
\[
\mathcal{K}:=\{(V_0,\ldots,V_n)\in\mathrm{Fl}(V):V_{\lambda_1^*+\ldots+\lambda_j^*}=\ker
x^j\ \mbox{ for all $j=1,\ldots,\lambda_1$}\}.
\]
Clearly $\mathcal{K}\subset\mathrm{Fl}_x(V)$. Moreover $\mathcal{K}$
is closed, isomorphic to the multiple flag variety
\[\mathrm{Fl}(\ker x)\times \mathrm{Fl}(\ker x^2/\ker x)\times\cdots\times \mathrm{Fl}(\ker x^{\lambda_1}/\ker x^{\lambda_1-1}),\]
hence irreducible, smooth, and of the same dimension as
$\mathrm{Fl}_x(V)$. Therefore $\mathcal{K}$ is a smooth irreducible
component of $\mathrm{Fl}_x(V)$. Since
$x\in\mathfrak{sl}_n(\mathbb{K})$ (nilpotent) is arbitrary, this
shows Theorem \ref{T1} in the case of $\mathrm{SL}_n(\mathbb{K})$.

\begin{remark}
\label{R4} Another way of proving Theorem \ref{T1} in the case of
$\mathrm{SL}_n(\mathbb{K})$ is to note that, for every nilpotent
element $x\in\mathfrak{sl}_n(\mathbb{K})$, its nilpotent orbit
$\mathcal{O}_x$ is of Richardson type, corresponding to the
parabolic subgroup $P$ formed by the elements
$g\in\mathrm{SL}_n(\mathbb{K})$ which fix the partial flag written
in (\ref{3}). Hence, in view of Remark~\ref{R2}, $\mathcal{O}_x$
contains at least one smooth orbital variety.
\end{remark}

\section{Some preliminaries on classical cases}

\label{section-4}

\subsection{A parametrization of components of $\mathrm{Fl}_x(V)$ by Young tableaux}

\label{section-4.1}

As in Section \ref{section-SLn}, let
$x\in\mathfrak{sl}_n(\mathbb{K})$ be a nilpotent endomorphism of
$V=\mathbb{K}^n$ and let $\lambda(x)=(\lambda_1,\ldots,\lambda_k)$
be the lengths of the Jordan blocks of $x$ listed in nonincreasing
order. The so-obtained partition $\lambda=\lambda(x)\vdash n$ can be
viewed as a Young diagram, i.e., as the set of $n$ empty boxes
displayed along left-justified rows of lengths
$\lambda_1,\ldots,\lambda_k$. A standard Young tableau $\tau$ of
shape $\lambda$ is by definition a numbering of the boxes of
$\lambda$ from $1$ to $n$ such that the numbers increase from left
to right along the rows and from top to bottom along the columns.
Equivalently, $\tau$ can be viewed as a maximal chain of Young
diagrams
\[\emptyset=\lambda^0(\tau)\subset\lambda^1(\tau)\subset \lambda^2(\tau)\subset\ldots\subset \lambda^n(\tau)=\lambda,\]
where $\lambda^i(\tau)$ denotes the shape of the subtableau formed
by the first $i$ entries of $\tau$.

It is known (see \cite{Spaltenstein}) that the components of
$\mathrm{Fl}_x(V)$ can be parameterized by the standard Young
tableaux of shape $\lambda(x)$. In fact there are several ways to
realize this parametrization, and the way (described below) which is
suitable for our purpose in this paper is somewhat different than
the one in \cite{Spaltenstein}.

%Let $c=(a_i,b_i)_{i=0}^n$ be a sequence of pairs of integers such
%that
%\[0=a_0\leq a_1\leq\ldots\leq a_n=b_n\leq \ldots\leq b_1\leq b_0=n\ \mbox{ and }\ b_i-a_i=n-i\mbox{ for all $i$.}\]
Let
$(a_i,b_i)_{i=0}^n:=(\lfloor\frac{i}{2}\rfloor,n-\lceil\frac{i}{2}\rceil)_{i=0}^n$.
In other words
\[\textstyle(a_i,b_i)_{i=0}^n=\big(
(0,n),(0,n-1),(1,n-1),(1,n-2),\ldots,(\lfloor\frac{n}{2}\rfloor,\lfloor\frac{n}{2}\rfloor)
\big)\] and we have $b_i-a_i=n-i$ for all $i=0,\ldots,n$. Set
\[\mathrm{Fl}_{x,\tau}(V)=\{(V_0,\ldots,V_n)\in\mathrm{Fl}_x(V):\lambda(x|_{V_{b_i}/V_{a_i}})=\lambda^{n-i}(\tau)\ \forall i=1,\ldots,n\},\]
where $x|_{V_{b}/V_{a}}$ stands for the nilpotent endomorphism
induced by $x$ on the subquotient $V_{b}/V_{a}$ and
$\lambda(x|_{V_{b}/V_{a}})$ stands for its Jordan form, seen as a
Young diagram.

\begin{proposition}[\protect{\cite[\S3.1--3.2]{Fresse-2009}}]
\label{P6} $\mathrm{Fl}_x(V)=\bigsqcup_\tau\mathrm{Fl}_{x,\tau}(V)$,
where the union is taken over all standard Young tableaux of shape
$\lambda(x)$. Moreover the subsets $\mathrm{Fl}_{x,\tau}(V)$ are
nonempty, locally closed, irreducible, smooth, all of the same
dimension. Therefore, the closures
$\overline{\mathrm{Fl}_{x,\tau}(V)}$ (for $\tau$ running over the
set of standard Young tableaux of shape $\lambda(x)$) are exactly
the irreducible components of $\mathrm{Fl}_x(V)$.
\end{proposition}

\begin{remark}
As shown in \cite[\S3.1--3.2]{Fresse-2009}, Proposition \ref{P6}
holds whenever $(a_i,b_i)_{i=0}^n$ is replaced by any double
sequence such that
\[0=a_0\leq a_1\leq\ldots\leq a_n=b_n\leq \ldots\leq b_1\leq b_0=n\ \mbox{ and }\ b_i-a_i=n-i\mbox{ for all $i$.}\]
This leads to several parametrizations of the components of
$\mathrm{Fl}_x(V)$. The original parametrization given in
\cite{Spaltenstein} corresponds to the choice of the sequence
$(0,n-i)_{i=0}^n$.
\end{remark}

\subsection{Some components of $\mathrm{Fl}_x(V)$ associated to domino tableaux}

\label{section-4.2}

As above, $\lambda=\lambda(x)\vdash n$ is the Jordan form of a nilpotent element $x\in\mathfrak{sl}_n(\mathbb{K})$. A (standard) domino tableau $d$ of shape $\lambda$ is by definition a numbering of the boxes of $\lambda$ by the numbers $0,1,\ldots,\lfloor\frac{n}{2}\rfloor$ such that
\begin{itemize}
\item every $i\in\{1,\ldots,\lfloor\frac{n}{2}\rfloor\}$ appears exactly twice in the numbering
while $0$ appears at most once (i.e., $0$ appears only if $n$ is odd);
\item the numbers are nondecreasing from left to right along the rows and from top to bottom along the columns;
\item for every $i\in\{1,\ldots,\lfloor\frac{n}{2}\rfloor\}$, the two boxes of number $i$ are adjacent, i.e., they belong to the same row, forming a horizontal domino {\scriptsize $\young(ii)$}, or to the same column, forming a vertical domino {\scriptsize $\young(i,i)$}.
\end{itemize}
The domino tableau $d$ induces a chain of Young diagrams
\begin{equation}
\label{4}
\lambda^0(d)\subset\lambda^1(d)\subset\lambda^2(d)\subset\ldots\subset\lambda^{\lfloor\frac{n}{2}\rfloor}(d)=\lambda
\end{equation}
where, for every $i\in\{0,\ldots,\lfloor\frac{n}{2}\rfloor\}$, we
denote by $\lambda^i(d)$ the shape of the subtableau of $d$ formed
by the boxes of number $\leq i$. Thus
$\lambda^{\lfloor\frac{n}{2}\rfloor-i}(d)$ has $n-2i$ boxes. The set
\[\mathrm{Fl}_{x,d}(V):=\{(V_0,\ldots,V_n)\in\mathrm{Fl}_x(V):\lambda(x|_{V_{n-i}/V_i})=\lambda^{\lfloor\frac{n}{2}\rfloor-i}(d)\ \forall \textstyle i=0,\ldots,\lfloor\frac{n}{2}\rfloor\}\]
is a well-defined subset of $\mathrm{Fl}_x(V)$.

\begin{proposition}
For every domino tableau $d$ of shape $\lambda=\lambda(x)$, the
subset $\mathrm{Fl}_{x,d}(V)$ is nonempty, locally closed, smooth,
and its closure is an irreducible component of $\mathrm{Fl}_x(V)$.
\end{proposition}

\begin{proof}
For every $i\in\{1,\ldots,\lfloor\frac{n}{2}\rfloor\}$, the Young
diagrams $\lambda^i(d)$ and $\lambda^{i-1}(d)$ differ by one couple
of adjacent boxes, hence there is a unique Young diagram
$\tilde{\lambda}^i(d)$ such that
$\lambda^{i-1}(d)\subsetneq\tilde{\lambda}^i(d)\subsetneq\lambda^i(d)$.
This means that there is a unique maximal chain of Young diagrams
which refines the chain written in (\ref{4}), and this implies that
there is a unique standard Young tableau $\tau(d)$ of shape
$\lambda$ such that
\[\lambda^{n-2i}(\tau(d))=\lambda^{\lfloor\frac{n}{2}\rfloor-i}(d)\ \mbox{ for all $\textstyle i=0,\ldots,\lfloor\frac{n}{2}\rfloor$}.\]
Comparing the definitions of $\mathrm{Fl}_{x,\tau}(V)$ and
$\mathrm{Fl}_{x,d}(V)$ and using the uniqueness of $\tau(d)$, we
deduce that $\mathrm{Fl}_{x,\tau(d)}(V)\subset\mathrm{Fl}_{x,d}(V)$
and $\mathrm{Fl}_{x,\tau}(V)\cap\mathrm{Fl}_{x,d}(V)=\emptyset$ for
all $\tau\not=\tau(d)$, whence
$\mathrm{Fl}_{x,\tau(d)}(V)=\mathrm{Fl}_{x,d}(V)$ (by Proposition
\ref{P6}). The result now follows from the properties of the
subvariety $\mathrm{Fl}_{x,\tau(d)}(V)$ stated in Proposition
\ref{P6}.
\end{proof}

\begin{example}
\label{E1} {\rm (a)} Assume that $x=0$, so
$\lambda(x)=(1,\ldots,1)$, i.e., the Young diagram
$\lambda=\lambda(x)$ consists of one column of size $n$. Then there
is only one domino tableau $d_{n,0}$ of shape $\lambda$, namely,
\[d_{n,0}:=\mbox{\scriptsize $\young(1,1,:,:,m,m)$}\ \mbox{ if $n=2m$ is even,}\quad d_{n,0}:=\mbox{\scriptsize $\young(0,1,1,:,:,m,m)$}\ \mbox{ if $n=2m+1$ is odd.}\]
In addition, in this case
$\mathrm{Fl}_{x,d_{n,0}}(V)=\mathrm{Fl}_x(V)=\mathrm{Fl}(V)$ is the
whole flag variety.

{\rm (b)} Assume that $n=2m$ is even and $x$ is of nilpotency order
two, thus $\lambda(x)=(2^k,1^{n-2k})$, i.e., the Young diagram
$\lambda=\lambda(x)$ has two columns of sizes $n-k$ and $k$,
respectively. An example of domino tableau of shape $\lambda$ is
given by
\[d_{n,k}:=\mbox{\scriptsize $\young(11,22,::,kk,:,:,m,m)$}.\]
For this domino tableau, we claim that
\begin{equation}
\label{5}
\overline{\mathrm{Fl}_{x,d_{n,k}}(V)}=\mathcal{K}_{n,k}:=\{(V_0,\ldots,V_n)\in\mathrm{Fl}(V):\mathrm{Im}\,x\subset
V_m\subset\ker x\},
\end{equation}
which is smooth. Note that $\mathcal{K}_{n,k}$ is a closed
subvariety of $\mathrm{Fl}_x(V)$, which is irreducible and smooth
since the map $(V_0,\ldots,V_n)\mapsto V_m/\mathrm{Im}\,x$ is a
locally trivial fiber bundle of $\mathcal{K}_{n,k}$ onto the
Grassmannian variety $\{W\subset\ker x/\mathrm{Im}\,x:\dim W=m-k\}$,
whose fiber is isomorphic to the double flag variety
$\mathrm{Fl}(\mathbb{K}^m)\times \mathrm{Fl}(\mathbb{K}^m)$. Hence,
for showing (\ref{5}), it suffices to show the inclusion
$\mathrm{Fl}_{x,d_{n,k}}(V)\subset\mathcal{K}_{n,k}$. So let
$(V_0,\ldots,V_n)\in\mathrm{Fl}_{x,d_{n,k}}(V)$. Thus each subspace
$V_i$ is $x$-stable and the nilpotent endomorphism induced by $x$ on
the subquotient $V_{m+i}/V_{m-i}$ has Jordan form $(2^i)$ for all
$i\in\{0,\ldots,k\}$. On one hand this property implies
$\mathrm{rk}\,x|_{V_{m+k}/V_{m-k}}=\mathrm{rk}\,x=k$, which forces
$V_{m-k}\subset\ker x$. On the other hand we get
\[
\ker x|_{V_{m+i}/V_{m-i}}=\mathrm{Im}\,x|_{V_{m+i}/V_{m-i}}\ \mbox{
for all $i=0,\ldots,k$.}
\]
Whence $V_{m-i+1}/V_{m-i}\subset\mathrm{Im}\,x|_{V_{m+i}/V_{m-i}}$,
which yields the inclusion $V_{m-i+1}\subset \mathrm{Im}\,x+V_{m-i}$
for all $i\in\{1,\ldots,k\}$. It follows that
$V_m\subset\mathrm{Im}\,x+V_{m-k}$ and in fact, by comparing the
dimensions, we have $V_m=\mathrm{Im}\,x+V_{m-k}$. Therefore
\[\mathrm{Im}\,x\subset V_m=\mathrm{Im}\,x+V_{m-k}\subset\mathrm{Im}\,x+\ker x=\ker x.\]
The inclusion $\mathrm{Fl}_{x,d_{n,k}}(V)\subset\mathcal{K}_{n,k}$
is shown.

 {\rm (c)} Assume that $n=6$ and $\lambda(x)=(2,2,1,1)$. Then
\[d:=\mbox{\scriptsize $\young(12,12,3,3)$}\]
is an example of domino tableau for which the irreducible component
$\overline{\mathrm{Fl}_{x,d}(V)}$ of $\mathrm{Fl}_x(V)$ is singular
(see Section \ref{section-6}).
\end{example}

\subsection{Flag varieties and Springer fibers in classical cases}

\label{section-4.3}

Hereafter we assume that the space $V=\mathbb{K}^n$ is endowed with
a nondegenerate bilinear form $\omega$ which can be orthogonal or
symplectic. Let $G=G(V,\omega)\subset\mathrm{SL}_n(\mathbb{K})$ be
the subgroup of elements which preserve $\omega$. Thus $G$ is
isomorphic to $\mathrm{SO}_n(\mathbb{K})$ if $\omega$ is orthogonal,
resp., to $\mathrm{Sp}_n(\mathbb{K})$ if $\omega$ is symplectic (in
which case $n$ is necessarily even). In both cases, the Lie algebra
$\mathfrak{g}=\mathfrak{g}(V,\omega)\subset\mathfrak{sl}_n(\mathbb{K})$
consists of all endomorphisms of trace zero which are skew-adjoint
with respect to $\omega$.

A nilpotent element $x\in\mathfrak{g}$ is then a nilpotent
endomorphism of $V$ which is skew-adjoint. Its Jordan form
$\lambda(x)=\lambda=(\lambda_1,\ldots,\lambda_k)$ satisfies (see
\cite[\S 1]{Jantzen}):
\begin{itemize}
\item in the case where $\omega$ is orthogonal: for every even
number $2\ell$, the Young diagram $\lambda$ has an even number of
rows of length $2\ell$ (i.e., $|\{i:\lambda_i=2\ell\}|$ is even);
\item in the case where $\omega$ is symplectic: for every odd number
$2\ell+1$, the Young diagram $\lambda$ has an even number of rows of
length $2\ell+1$ (i.e., $|\{i:\lambda_i=2\ell+1\}|$ is even).
\end{itemize}
We say that a partition $\lambda$ (equivalently, a Young diagram) is
$\omega$-admissible whenever this parity condition holds.

Recall that $\mathcal{B}=G/B$ denotes the flag variety of $G$, where
$B\subset G$ is some (any) Borel subgroup, and
$\mathcal{B}_x\subset\mathcal{B}$ is the Springer fiber associated
to $x$ (see Section \ref{section-1.1}). It is well known that
$\mathcal{B}$ and $\mathcal{B}_x$ can be viewed as varieties of
isotropic complete flags, as follows (see, e.g., \cite{Jantzen} for
more details). Let
$\sigma=\sigma(\omega):\mathrm{Fl}(V)\to\mathrm{Fl}(V)$ be the
involution defined by letting
\[\sigma(V_0,V_1,\ldots,V_n)=(V_n^\perp,\ldots,V_1^\perp,V_0^\perp),\]
where $V_i^\perp$ stands for the orthogonal of $V_i$ with respect to
the form $\omega$. Set
\[\mathrm{Fl}(V,\omega):=(\mathrm{Fl}(V))^\sigma=\{(V_0,\ldots,V_n)\in\mathrm{Fl}(V):V_i^\perp=V_{n-i}\ \forall i=0,\ldots,n\}.\]
The property of $x$ of being skew-adjoint guarantees that
$\mathrm{Fl}_x(V)$ is $\sigma$-stable. Let
\[
\mathrm{Fl}_x(V,\omega):=(\mathrm{Fl}_x(V))^\sigma=\{(V_0,\ldots,V_n)\in\mathrm{Fl}(V,\omega):x(V_i)\subset
V_i\ \forall i=0,\ldots,n\}.
\]

\begin{proposition}
\label{P8} $G=G(V,\omega)$ acts on $\mathrm{Fl}(V,\omega)$ in a
natural way and for every
$F=(V_0,\ldots,V_n)\in\mathrm{Fl}(V,\omega)$ the subgroup
$\{g\in G:gF=F\}$
is a Borel subgroup of
$G$. Moreover:
\begin{itemize}
\item[\rm (a)] If $n$ is odd or $\omega$ is symplectic,
then $\mathrm{Fl}(V,\omega)$ is $G$-homogeneous (in particular
connected), hence isomorphic to $\mathcal{B}$; this isomorphism
restricts to a $Z_G(x)$-equivariant isomorphism between the
subvarieties $\mathrm{Fl}_x(V,\omega)$ and $\mathcal{B}_x$.
\item[\rm (b)] If $n$ is even and $\omega$ is orthogonal,
then $\mathrm{Fl}(V,\omega)$ consists of exactly two
$G$-homogenenous connected components, which are thus both
isomorphic to~$\mathcal{B}$; accordingly, $\mathrm{Fl}_x(V,\omega)$
splits into two connected components, both isomorphic to
$\mathcal{B}_x$ (through $Z_G(x)$-equivariant isomorphisms).
\end{itemize}
\end{proposition}

Let us make the statement of Proposition \ref{P8}\,{\rm (b)}
slightly more precise. Assume that $n=2m$ is even and $\omega$ is
orthogonal. For every $(m-1)$-dimensional isotropic subspace $W$,
there are exactly two Lagrangian subspaces which contain $W$ (since
it is so in the two-dimensional space $W^\perp/W$ endowed with the
restriction of $\omega$). This yields a well-defined involution
$\iota:\mathrm{Fl}(V,\omega)\to\mathrm{Fl}(V,\omega)$ given by
\begin{equation}
\label{iota} \iota:(V_0,\ldots,V_n)\mapsto
(V_0,\ldots,V_{m-1},\widetilde{V_m},V_{m+1},\ldots,V_n)
\end{equation}
where $\widetilde{V_m}$ is the unique Lagrangian subspace containing
$V_{m-1}$ (hence contained in $V_{m+1}=V_{m-1}^\perp$) such that
$\widetilde{V_m}\not=V_m$. The map $\iota$ is in fact algebraic,
$G$-equivariant, and it maps one connected component of
$\mathrm{Fl}(V,\omega)$ onto the other.

If the $(m-1)$-dimensional isotropic subspace $W$ is in addition
stable by $x$, then so are the two Lagrangian subspaces which
contain $W$ (since the endomorphism induced by $x$ on the
subquotient $W^\perp/W$ is trivial). This observation shows that the
subvariety $\mathrm{Fl}_x(V,\omega)$ is $\iota$-stable, hence it
intersects both connected components of $\mathrm{Fl}(V,\omega)$.

\subsection{Components of Springer fibers in classical cases}

\label{section-4.4}

The notation is as in Section \ref{section-4.3}. In this section, we
recall from \cite{van-Leeuwen} the parametrization of the
irreducible components of $\mathrm{Fl}_x(V,\omega)$ by so-called
admissible domino tableaux. First note that, for every
$(V_0,\ldots,V_n)\in\mathrm{Fl}_x(V)$, we have
\[V_{n-i}^\perp+x^\ell(V_i^\perp)=(V_{n-i}\cap (x^\ell)^{-1}(V_i))^\perp\ \mbox{ for all $\ell\geq 1$},\]
thus
\begin{eqnarray*}
\dim\mathrm{Im}\,x^\ell|_{V_i^\perp/V_{n-i}^\perp} & = & \dim
(V_{n-i}^\perp+x^\ell(V_i^\perp))-i \\
 & = & n-i- \dim V_{n-i}\cap (x^\ell)^{-1}(V_i) \\
 & = & n-2i-\dim\ker x^\ell|_{V_{n-i}/V_i} \\
 & = & \dim\mathrm{Im}\,x^\ell|_{V_{n-i}/V_i},
\end{eqnarray*}
hence the Jordan forms of $x|_{V_{n-i}/V_i}$ and
$x|_{V_i^\perp/V_{n-i}^\perp}$ coincide for all
$i=0,\ldots,\lfloor\frac{n}{2}\rfloor$, which ensures that the set
$\mathrm{Fl}_{x,d}(V)$ is $\sigma$-stable whenever $d$ is a domino
tableau of shape $\lambda=\lambda(x)$. We set
\begin{eqnarray*}
\mathrm{Fl}_{x,d}(V,\omega)\!\! & := &
\!\!(\mathrm{Fl}_{x,d}(V))^\sigma
\\
 & = & \textstyle \!\!\{(V_0,\ldots,V_n)\in\mathrm{Fl}_x(V,\omega):\lambda(x|_{V_i^\perp/V_i})=\lambda^{n-i}(d)\
\forall i=0,\ldots,\lfloor\frac{n}{2}\rfloor\}. \end{eqnarray*}

We say that a domino tableau $d$ is $\omega$-admissible if, for
every $i$, the shape of the subtableau formed by the boxes of number
$\leq i$ is an $\omega$-admissible Young diagram, in the sense of
Section \ref{section-4.3}.

For $(V_0,\ldots,V_n)\in\mathrm{Fl}_x(V,\omega)$ and every $i$,
the nilpotent endomorphism $x|_{V_i^\perp/V_i}$ is skew-adjoint with
respect to the form induced by $\omega$ on the subquotient
$V_i^\perp/V_i$, hence the Young diagram
$\lambda(x|_{V_i^\perp/V_i})$ is $\omega$-admissible (see Section
\ref{section-4.3}). Thus
\[\mathrm{Fl}_{x,d}(V,\omega)\not=\emptyset\ \mbox{ only if $d$ is $\omega$-admissible}.\]
In fact, letting $i$ run over
$\{0,\ldots,\lfloor\frac{n}{2}\rfloor\}$, we get a chain of
$\omega$-admissible Young diagrams
\[\lambda(x|_{V_{\lfloor\frac{n}{2}\rfloor}^\perp/V_{\lfloor\frac{n}{2}\rfloor}})\subset\ldots\subset\lambda(x|_{V_1^\perp/V_1})\subset\lambda(x|_{V_0^\perp/V_0})\]
and the $\omega$-admissibility condition implies that any two
consecutive diagrams in this chain differ by a couple of adjacent
boxes. Hence there is a domino tableau $d$ such that
$(V_0,\ldots,V_n)\in\mathrm{Fl}_{x,d}(V,\omega)$. Therefore,
\begin{equation}
\label{partition}
\mathrm{Fl}_x(V,\omega)=\bigsqcup_d\mathrm{Fl}_{x,d}(V,\omega)
\end{equation} where the union is taken over all $\omega$-admissible domino
tableaux.

\begin{proposition}[\protect{\cite{van-Leeuwen}}]
\label{P9}
Whenever $d$ is $\omega$-admissible, the subset
$\mathrm{Fl}_{x,d}(V,\omega)$ is non\-empty, locally closed,
equidimensional, of the same dimension as $\mathrm{Fl}_x(V,\omega)$.
Thus, each irreducible component of
$\overline{\mathrm{Fl}_{x,d}(V,\omega)}$ is an irreducible component
of $\mathrm{Fl}_x(V,\omega)$, and every component of
$\mathrm{Fl}_x(V,\omega)$ can be obtained in this way.
\end{proposition}

The results in \cite{van-Leeuwen} are in fact much more precise:
they include an explicit parametrization of the irreducible
components of each $\overline{\mathrm{Fl}_{x,d}(V,\omega)}$, in
terms of some equivalence classes of so-called signed domino
tableaux. In particular, the action of the group $Z_G(x)$ on this
set of components is explicitly described (note that
$\mathrm{Fl}_{x,d}(V,\omega)$ and so
$\overline{\mathrm{Fl}_{x,d}(V,\omega)}$ are $Z_G(x)$-stable).

We also mention that a combinatorial parametrization of the orbital
varieties of the nilpotent $G$-orbit $\mathcal{O}_x$ in terms of
domino tableaux is shown in \cite{McGovern}. This parametrization is
somewhat different from the one in \cite{van-Leeuwen} (i.e., the
correspondence of Proposition \ref{P1} does not match the two
parametrizations) since it involves unsigned domino tableaux, which
are not necessarily $\omega$-admissible in the above sense. See also
\cite{Pietraho} for a comparison of these two parametrizations.

%\subsection{Outline of proof of Theorem \ref{T1}}

\medskip

In light of the facts described in this section, we make the
following conclusion:

\begin{proposition}
\label{P10} To prove Theorem \ref{T1}, it suffices to show that,
for every nilpotent element
$x\in\mathfrak{g}=\mathfrak{g}(V,\omega)$, there is an
$\omega$-admissible domino tableau $d$ of shape $\lambda(x)$ such
that
\begin{itemize}
\item[\rm (a)] the irreducible component
$\overline{\mathrm{Fl}_{x,d}(V)}$ of $\mathrm{Fl}_x(V)$ is smooth;
\item[\rm (b)] the subvariety
$\overline{\mathrm{Fl}_{x,d}(V,\omega)}$ consists of one irreducible
component if $n$ is odd or $\omega$ is symplectic, respectively of
two irreducible components if $n$ is even and $\omega$ is
orthogonal.
\end{itemize}
\end{proposition}

\begin{proof}
Assume that we have found $d$ satisfying conditions {\rm (a)} and
{\rm (b)}. Since the subset $\mathrm{Fl}_{x,d}(V)$ is stable by the
involution $\sigma$, so is its closure
$\overline{\mathrm{Fl}_{x,d}(V)}$. Proposition \ref{P3}, applied to
the variety $X=\overline{\mathrm{Fl}_{x,d}(V)}$ (which is smooth, by
{\rm (a)}) and the group $H=\{\mathrm{id},\sigma\}$, implies that
the fixed-point subvariety
$(\overline{\mathrm{Fl}_{x,d}(V)})^\sigma$ is smooth.

The inclusions
\[\overline{\mathrm{Fl}_{x,d}(V,\omega)}=\overline{(\mathrm{Fl}_{x,d}(V))^\sigma}\subset(\overline{\mathrm{Fl}_{x,d}(V)})^\sigma\subset(\mathrm{Fl}_x(V))^\sigma=\mathrm{Fl}_x(V,\omega)\]
combined with Proposition \ref{P9} show that $\overline{\mathrm{Fl}_{x,d}(V,\omega)}$ is a union of irreducible components of $(\overline{\mathrm{Fl}_{x,d}(V)})^\sigma$.
Hence the components of $\overline{\mathrm{Fl}_{x,d}(V,\omega)}$ are pairwise disjoint and smooth.

As already noted, the subvariety $\overline{\mathrm{Fl}_{x,d}(V,\omega)}$ (as well as its set of components) is $Z_G(x)$-stable.

In the case where $n$ is odd or $\omega$ is symplectic,
condition {\rm (b)} implies that $\overline{\mathrm{Fl}_{x,d}(V,\omega)}$
is a smooth, $Z_G(x)$-stable irreducible component of $\mathrm{Fl}_x(V,\omega)$.

In the case where $n$ is even and $\omega$ is orthogonal,
the variety $\mathrm{Fl}_x(V,\omega)$ itself consists of two $Z_G(x)$-stable connected components
$C_1$ and $C_2$
(see Proposition \ref{P8}). Moreover $\mathrm{Fl}_x(V,\omega)$ is endowed with the involution
$\iota$ defined in (\ref{iota}), which is such that $\iota(C_1)=C_2$.
The definition of $\mathrm{Fl}_{x,d}(V,\omega)$ guarantees that
$\mathrm{Fl}_{x,d}(V,\omega)$ and so $\overline{\mathrm{Fl}_{x,d}(V,\omega)}$ are $\iota$-stable.
Condition (b) then implies that $\overline{\mathrm{Fl}_{x,d}(V,\omega)}$
has exactly two smooth, $Z_G(x)$-stable components, namely $\overline{\mathrm{Fl}_{x,d}(V,\omega)}\cap C_1$ and $\overline{\mathrm{Fl}_{x,d}(V,\omega)}\cap C_2$.

In both cases,
Propositions \ref{P8} and \ref{P9} allow us to conclude that the Springer fiber $\mathcal{B}_x$ has a smooth, $Z_G(x)$-stable irreducible component, as claimed in Theorem \ref{T1}.
\end{proof}

%\section{Proof of Theorem \ref{T1} in the cases $G=\mathrm{SO}_n(\mathbb{K})$ and $G=\mathrm{Sp}_{2n}(\mathbb{K})$}

\section{Construction of a smooth component of $\mathrm{Fl}_x(V)$}

\label{section-5}

\subsection{A combinatorial construction}

\label{section-5.1}

Let $n_1,n_2$ be positive integers, with $n_2$ even. Let $\lambda^1$
and $\lambda^2$ be partitions of $n_1$ and $n_2$, respectively, seen
as Young diagrams. We make the following assumption:
\[\mbox{(length of the last column of $\lambda^1$)}\geq\mbox{(length of the first column of $\lambda^2$).}\]
Under this assumption, if $d_1$ and $d_2$ are domino tableaux of
respective shapes $\lambda^1$ and $\lambda^2$, we can define a
domino tableau $d$ obtained by concatenation. Specifically, denoting
by $\ell_1$ and $\ell_2$ the number of columns of $d_1$ and $d_2$
respectively, $d=:d_1+d_2$ is the domino tableau such that:
\begin{itemize}
\item for every $i\in\{1,\ldots,\ell_1\}$, the $i$-th column of $d$
coincides with the $i$-th column of $d_1$;
\item for every $i\in\{\ell_1+1,\ldots,\ell_1+\ell_2\}$, the $i$-th
column of $d$ is obtained from the $(i-\ell_1)$-th column of $d_2$
by adding $\lfloor\frac{n_1}{2}\rfloor$ to each box number.
\end{itemize}
In particular, the shape of $d$ is the Young diagram
$\lambda=:\lambda^1+\lambda^2$ of size $n_1+n_2$ obtained by
juxtaposing the two diagrams $\lambda^1$ and $\lambda^2$.

\begin{example}
$\mbox{\scriptsize $\young(011,235,235,466,4)$}+\mbox{\scriptsize
$\young(11,22,3,3)$}=\mbox{\scriptsize
$\young(01177,23588,2359,4669,4)$}$.
\end{example}

In Section \ref{section-4.4}, we consider the following parity
conditions for domino tableaux:
\begin{description}
\item[\rm (O) (resp., (S))] for every $i$ and $\ell$, the subtableau formed by
the boxes of number $\leq i$ has an even number of rows of length
$2\ell$ (resp., $2\ell+1$).
\end{description}
In other words, $d$ satisfies (O) (resp., (S)) if and only if it is
$\omega$-admissible for an orthogonal (resp., a symplectic) form
$\omega$, in the sense of Section \ref{section-4.4}.

\begin{lemma}
\label{L1} Let $d_1,d_2$ be two domino tableaux whose concatenation
$d:=d_1+d_2$ is well defined. Assume that $d_2$ satisfies {\rm (S)}.
Then:
\begin{itemize}
\item[\rm (a)] if $d_1$ satisfies {\rm (O)} and has an odd number of
columns, then $d$ satisfies {\rm (O)};
\item[\rm (b)] if $d_1$ satisfies {\rm (S)} and has an even number of
columns, then $d$ satisfies {\rm (S)}.
\end{itemize}
\end{lemma}

\begin{proof}
As above, we respectively denote by $n_1$ and $\ell_1$ the number of
boxes and the number of columns in $d_1$. The definition of
$d=d_1+d_2$ implies that
\[\textstyle n(d,i,\ell)=n(d_1,i,\ell)+n(d_2,i-\lfloor\frac{n_1}{2}\rfloor,\ell-\ell_1)\ \mbox{ for all $i\geq 0$, all $\ell\not=\ell_1$},\]
where $n(\cdot,i,\ell)$ stands for the number of rows of length
$\ell$ in the subtableau formed by the boxes of number $\leq i$. In
the situation {\rm (a)} of the statement, for every $i\geq 0$ and
every even integer $\ell$, the number $\ell-\ell_1$ is odd (so in
particular $\ell\not=\ell_1$), the numbers $n(d_1,i,\ell)$ and
$n(d_2,i-\lfloor\frac{n_1}{2}\rfloor,\ell-\ell_1)$ are even (since
$d_1$ and $d_2$ satisfy (O) and (S), respectively), hence the above
equality implies that $n(d,i,\ell)$ is even. This shows that $d$
satisfies (O). In {\rm (b)}, for every $i\geq 0$ and every odd
integer $\ell$, the number $\ell-\ell_1$ is still odd,
$n(d_1,i,\ell)$ and
$n(d_2,i-\lfloor\frac{n_1}{2}\rfloor,\ell-\ell_1)$ are even (since
$d_1$ and $d_2$ satisfy (S)), hence $n(d,i,\ell)$ is even. Therefore
$d$ satisfies (S).
\end{proof}

As in Sections \ref{section-4.3}--\ref{section-4.4}, the space
$V=\mathbb{K}^n$ is endowed with the form $\omega$ (orthogonal or
symplectic), $x\in\mathfrak{g}=\mathfrak{g}(V,\omega)$ is a
nilpotent element, i.e., a nilpotent endomorphism of $V$ which is
skew adjoint with respect to $\omega$, and $\lambda(x)\vdash n$
stands for the Jordan form of $x$, seen as a Young diagram. In
particular this Young diagram is $\omega$-admissible, in the sense
of Section \ref{section-4.3}. Here we construct an
$\omega$-admissible domino tableau $d_x^\omega$ of shape
$\lambda(x)$, that is, a domino tableau satisfying (O) if $\omega$
is orthogonal, respectively (S) if $\omega$ is symplectic.

\begin{notation}
\label{N1} Let $\lambda_1^*\geq \ldots\geq \lambda_\ell^*$ be the
lengths of the columns of $\lambda(x)$. Up to adding a term
$\lambda_{\ell+1}^*$ equal to zero, we may assume that $\ell$ is odd
when $\omega$ is orthogonal and $\ell$ is even when $\omega$ is
symplectic (hence $\lambda_{\ell'}^*\geq 1$ for all
$\ell'\not=\ell$, $\lambda_\ell^*\geq 0$). Recall the domino
tableaux $d_{n,k}$ introduced in Example \ref{E1}\,{\rm (a)}--{\rm
(b)} and note that $d_{n,k}$ satisfies {\rm (S)} whenever $n$ is
even and $k\geq 0$ while $d_{n,0}$ satisfies {\rm (O)} for all $n$.

\begin{itemize}
\item First assume that $\omega$ is orthogonal. Thus $\ell$ is odd, say $\ell=2k+1$. We view $\lambda(x)$ as the juxtaposition of a diagram with one
column of length $\lambda_1^*$ and $k$ diagrams with two columns of
lengths
$(\lambda_2^*,\lambda_3^*),\ldots,(\lambda_{2k}^*,\lambda_{2k+1}^*)$,
respectively. The $\omega$-admissibility of $\lambda(x)$ guarantees
that $\lambda_{2j}^*-\lambda_{2j+1}^*$ and so
$\lambda_{2j}^*+\lambda_{2j+1}^*$ are even for all $j$. Set
\begin{equation}
\label{7}
d_x^\omega=d_{\lambda_1^*,0}+d_{\lambda_2^*+\lambda_3^*,\lambda_3^*}+\ldots+d_{\lambda_{\ell-1}^*+\lambda_\ell^*,\lambda_\ell^*}
\end{equation}
(note that the concatenation procedure is clearly associative
whenever it is well defined). By Lemma \ref{L1}, the domino tableau
$d_x^\omega$ is $\omega$-admissible.
\item Next assume that $\omega$ is symplectic. Thus $\ell$ is even, say $\ell=2k$. Here we view $\lambda(x)$ as the juxtaposition of $k$ Young
diagrams with two columns of lengths
$(\lambda_1^*,\lambda_2^*),\ldots,(\lambda_{2k-1}^*,\lambda_{2k}^*)$,
respectively. Since $\lambda(x)$ is $\omega$-admissible, the numbers
$\lambda_{2j-1}^*-\lambda_{2j}^*$ and
$\lambda_{2j-1}^*+\lambda_{2j}^*$ are even for all $j$. Set
\begin{equation}
\label{8}
d_x^\omega=d_{\lambda_1^*+\lambda_2^*,\lambda_2^*}+\ldots+d_{\lambda_{\ell-1}^*+\lambda_\ell^*,\lambda_\ell^*}.
\end{equation}
Lemma \ref{L1} ensures that the so-obtained domino tableau
$d_x^\omega$ is $\omega$-admissible.
\end{itemize}
\end{notation}

\begin{example}
{\rm (a)} Assume that $\omega$ is orthogonal and
$\lambda(x)=(5,4^2,2^2)=\mbox{\scriptsize $\yng(5,4,4,2,2)$}$. We
write $\lambda(x)=\mbox{\scriptsize
$\yng(1,1,1,1,1)$}+\mbox{\scriptsize
$\yng(2,2,2,1,1)$}+\mbox{\scriptsize $\yng(2,1,1)$}$ and get
$d_x^\omega=\mbox{\scriptsize $\young(0,1,1,2,2)$}+\mbox{\scriptsize
$\young(11,22,33,4,4)$}+\mbox{\scriptsize
$\young(11,2,2)$}=\mbox{\scriptsize
$\young(03377,1448,1558,26,26)$}$.

{\rm (b)} Assume now that $\omega$ is symplectic and
$\lambda(x)=(5,5,4,1,1)=\mbox{\scriptsize $\yng(5,5,4,1,1)$}$, i.e.,
$\lambda(x)=\mbox{\scriptsize $\yng(2,2,2,1,1)$}+\mbox{\scriptsize
$\yng(2,2,2)$}+\mbox{\scriptsize $\yng(1,1)$}$\,. Thus
$d_x^\omega=\mbox{\scriptsize
$\young(11,22,33,4,4)$}+\mbox{\scriptsize
$\young(11,22,33)$}+\mbox{\scriptsize
$\young(1,1)$}=\mbox{\scriptsize $\young(11558,22668,3377,4,4)$}$ in
this case.
\end{example}

Following the terminology of \cite[\S3.3]{van-Leeuwen}, the domino
tableau $d_x^\omega$ introduced in Notation \ref{N1} always consists of a
single cluster. By \cite[Lemma
3.3.3]{van-Leeuwen}, this yields:

\begin{proposition}
\label{P11} As above, $x\in\mathfrak{g}(V,\omega)$ is a nilpotent
element. Let $d_x^\omega$ be as in (\ref{7}) or (\ref{8}) depending
on whether $\omega$ is orthogonal or symplectic.
\begin{itemize}
\item[\rm (a)]
If $\omega$ is symplectic or $n$ is odd, then the variety
$\overline{\mathrm{Fl}_{x,d_x^\omega}(V,\omega)}$ is irreducible.
\item[\rm (b)]
If $\omega$ is orthogonal and $n$ is even, then
$\overline{\mathrm{Fl}_{x,d_x^\omega}(V,\omega)}$ has exactly two
irreducible components.
\end{itemize}
\end{proposition}

\subsection{Structure of components of $\mathrm{Fl}_x(V)$ associated to domino tableaux obtained by concatenation}

\label{section-5.2}

In this section we focus on the variety $\mathrm{Fl}_x(V)$ of (non
necessarily isotropic) complete flags of $V=\mathbb{K}^n$ which are
stable by a given (non necessarily skew-adjoint) nilpotent
endomorphism $x:V\to V$. (In fact, for the moment we disregard the
bilinear form $\omega$.)

Let $d$ be a domino tableau of shape $\lambda(x)$ and let us
consider the irreducible component
$\overline{\mathrm{Fl}_{x,d}(V)}\subset\mathrm{Fl}_x(V)$ associated
to $d$ in the sense of Section \ref{section-4.2}. In the next
statement, the notation $\mathrm{Gr}_k(M)$ stands for the
Grassmannian variety formed by the $k$-dimensional subspaces of a
vector space $M$.

\begin{proposition}
\label{P12} Assume that $d$ is obtained as the concatenation of two
domino tableaux $d_1$ and $d_2$ of sizes $n_1$ and $n_2$,
respectively (thus $n=n_1+n_2$). Assume that $n_2=2m_2$ is even and
that we have in fact $d_2=d_{n_2,k_2}$ for some
$k_2\in\{0,\ldots,m_2\}$ (see Example \ref{E1}\,{\rm (a)}--{\rm
(b)}). Let $x_1\in\mathfrak{sl}_{n_1}(\mathbb{K})$ be a nilpotent
element whose Jordan form coincides with the shape of $d_1$, so that
$d_1$ gives rise to an irreducible component
$\overline{\mathrm{Fl}_{x_1,d_1}(\mathbb{K}^{n_1})}\subset\mathrm{Fl}_{x_1}(\mathbb{K}^{n_1})$.
\begin{itemize}
\item[\rm (a)] The following implication holds:
\[\mbox{$\overline{\mathrm{Fl}_{x_1,d_1}(\mathbb{K}^{n_1})}$ is
smooth}\ \Rightarrow\ \mbox{$\overline{\mathrm{Fl}_{x,d}(V)}$ is
smooth}.\]
\item[\rm (b)] More precisely, denoting by $\ell_1$ the number of
columns in $d_1$, we have
\[\mathrm{Im}\,x^{\ell_1+1}\subset V_{m_2}\subset\ker x\cap\mathrm{Im}\,x^{\ell_1}\ \mbox{ for all $(V_0,\ldots,V_n)\in\overline{\mathrm{Fl}_{x,d}(V)}$}\]
and the map
\[\begin{array}{rcl}
\Theta:\quad \overline{\mathrm{Fl}_{x,d}(V)} & \to &
\mathrm{Gr}_{m_2-k_2}(\ker x\cap\mathrm{Im}\,x^{\ell_1}/\mathrm{Im}\,x^{\ell_1+1}) \\
(V_0,\ldots,V_n) & \mapsto & V_{m_2}/\mathrm{Im}\,x^{\ell_1+1}
\end{array}\] is a well-defined, algebraic, locally trivial fiber bundle,
whose typical fiber is isomorphic to
\[\mathrm{Fl}(\mathbb{K}^{m_2})\times\mathrm{Fl}(\mathbb{K}^{m_2})\times\overline{\mathrm{Fl}_{x_1,d_1}(\mathbb{K}^{n_1})}.\]
\end{itemize}
\end{proposition}

The proof (given below) relies on the following lemma.

\begin{lemma}
\label{L2} Assume that $d$ is obtained as the concatenation of two
domino tableaux $d_1$ and $d_2$ of sizes $n_1$ and $n_2$,
respectively (thus $n=n_1+n_2$). Assume that $n_2=2m_2$ is even.
\begin{itemize}
\item[\rm (a)]
Let $\ell_1$ be the number of columns in $d_1$. For every
$(V_0,\ldots,V_n)\in\overline{\mathrm{Fl}_{x,d}(V)}$, we have
\[
x^{\ell_1}(V_{n-m_2})=V_{m_2}\quad\mbox{and}\quad
V_{n-m_2}=(x^{\ell_1})^{-1}(V_{m_2}).\] Hence, letting
$x_2=x|_{\mathrm{Im}\,x^{\ell_1}}$, the map
$\varphi:\overline{\mathrm{Fl}_{x,d}(V)}\to\mathrm{Fl}_{x_2}(\mathrm{Im}\,x^{\ell_1})$
given by \[\varphi:(V_0,\ldots,V_n)\mapsto
(V_0,\ldots,V_{m_2},x^{\ell_1}(V_{n-m_2+1}),\ldots,x^{\ell_1}(V_{n}))\]
is well defined (and algebraic).
\item[\rm (b)]
The Jordan form of $x_2$ coincides with the shape of the domino
tableau $d_2$, which gives rise to the subset
$\mathrm{Fl}_{x_2,d_2}(\mathrm{Im}\,x^{\ell_1})\subset\mathrm{Fl}_{x_2}(\mathrm{Im}\,x^{\ell_1})$;
in fact, we have
\[\varphi(\mathrm{Fl}_{x,d}(V))\subset \mathrm{Fl}_{x_2,d_2}(\mathrm{Im}\,x^{\ell_1})\]
and so
\[\varphi(\overline{\mathrm{Fl}_{x,d}(V)})\subset \overline{\mathrm{Fl}_{x_2,d_2}(\mathrm{Im}\,x^{\ell_1})}.\]
\item[\rm (c)]
Let $F=(V_0,\ldots,V_{n_2})\in
\mathrm{Fl}_{x_2,d_2}(\mathrm{Im}\,x^{\ell_1})$. Let $x_1=x_1(F)$ be
the nilpotent endomorphism of
$W(F):=(x^{\ell_1})^{-1}(V_{m_2})/V_{m_2}$ induced by $x$ (which
depends on $F$). Then, the Jordan form of $x_1$ coincides with the
shape of $d_1$, which gives rise to the subset
$\mathrm{Fl}_{x_1,d_1}(W(F))\subset\mathrm{Fl}_{x_1}(W(F))$; and in
fact, for every
$(W_0,\ldots,W_{n_1})\in\mathrm{Fl}_{x_1,d_1}(W(F))$, the element
\[(V_0,\ldots,V_{m_2},\pi^{-1}(W_1),\ldots,\pi^{-1}(W_{n_1}),(x^{\ell_1})^{-1}(V_{m_2+1}),\ldots,(x^{\ell_1})^{-1}(V_{n_2}))\]
belongs to $\mathrm{Fl}_{x,d}(V)$, where $\pi=\pi(F)$ stands
for the natural surjection $\pi:(x^{\ell_1})^{-1}(V_{m_2})\to
W(F)$.
\end{itemize}
\end{lemma}

In the proof of Lemma \ref{L2}, we use the following basic fact of
linear algebra.

\begin{lemma}
\label{L-linalg} Let $x:V\to V$ be a nilpotent endomorphism, whose
Jordan form $\lambda(x)$ is seen as a Young diagram. Let $\lambda^1$
be the subdiagram formed by the first $\ell_1$ columns of
$\lambda(x)$ and let $\lambda^2$ be the subdiagram formed by the
remaining columns, so that $\lambda(x)=\lambda^1+\lambda^2$ (with
the notation of Section \ref{section-5.1}). Then, for every
$x$-stable subspace $M\subset\mathrm{Im}\,x^{\ell_1}$, we have
\[\lambda^1=\lambda(x|_{(x^{\ell_1})^{-1}(M)/M})\quad \mbox{and}\quad
\lambda^2=\lambda(x|_{\mathrm{Im}\,x^{\ell_1}}),\] where
$x|_{(x^{\ell_1})^{-1}(M)/M}:(x^{\ell_1})^{-1}(M)/M\to(x^{\ell_1})^{-1}(M)/M$
and $x|_{\mathrm{Im}\, x^{\ell_1}}:\mathrm{Im}\, x^{\ell_1}\to
\mathrm{Im}\, x^{\ell_1}$ are the nilpotent maps induced by
$x$.
\end{lemma}

\begin{proof}[Proof of Lemma \ref{L-linalg}]
For every $k\geq 1$, the length of the $k$-th column of $\lambda(x)$
coincides with the dimension of $\ker x^k/\ker x^{k-1}$. Set
$x_1=x_1(M)=x|_{(x^{\ell_1})^{-1}(M)/M}$ and $x_2=x|_{\mathrm{Im}\,
x^{\ell_1}}$.

By construction $x_1^{\ell_1}=0$, hence $\lambda(x_1)$ has (at most)
$\ell_1$ columns. Moreover, for every $k\in\{1,\ldots,\ell_1\}$,
since $M\subset\mathrm{Im}\,x^{\ell_1}\subset\mathrm{Im}\,x^k$, we
have an isomorphism
\[(x^k)^{-1}(M)/\ker
x^k\stackrel{\sim}{\to}(x^{k-1})^{-1}(M)/\ker x^{k-1}\] induced by
$x$, hence
\[\dim \ker x_1^k/\ker x_1^{k-1}=\dim (x^k)^{-1}(M)/(x^{k-1})^{-1}(M)=
\dim \ker x^k/\ker x^{k-1}.\] Whence $\lambda(x_1)=\lambda^1$.

The map $x^{\ell_1}:V\to\mathrm{Im}\,x^{\ell_1}$ yields an
isomorphism $\ker x^{\ell_1+k}/\ker x^{\ell_1}\cong\ker x_2^k$ for
all $k\geq 1$. Whence $\lambda(x_2)=\lambda^2$.
\end{proof}

\begin{proof}[Proof of Lemma \ref{L2}]
{\rm (a)} Let $(V_0,\ldots,V_n)\in\mathrm{Fl}_{x,d}(V)$. By
definition of $\mathrm{Fl}_{x,d}(V)$, the nilpotent endomorphism
$x|_{V_{n-m_2}/V_{m_2}}$ induced by $x$ on the subquotient
$V_{n-m_2}/V_{m_2}$ has a Jordan form which coincides with the shape
of the subtableau $d_1$ of $d$. Since $d_1$ has $\ell_1$ columns,
this implies that $(x|_{V_{n-m_2}/V_{m_2}})^{\ell_1}=0$, hence
\begin{equation}
\label{L-1} x^{\ell_1}(V_{n-m_2})\subset V_{m_2}.
\end{equation}
Since (\ref{L-1}) is a closed relation, it holds more generally
whenever $(V_0,\ldots,V_n)\in\overline{\mathrm{Fl}_{x,d}(V)}$. Then,
the rank formula applied to the map $x^{\ell_1}:V_{n-m_2}\to
V_{m_2}$ yields
\begin{eqnarray*}
m_2 = \dim V_{m_2} & \geq & \dim x^{\ell_1}(V_{n-m_2}) = \dim
V_{n-m_2}-\dim V_{n-m_2}\cap\ker x^{\ell_1} \\
 & \geq & \dim V_{n-m_2}-\dim \ker x^{\ell_1} = (n-m_2)-n_1=m_2.
\end{eqnarray*}
This forces $\dim V_{m_2}=\dim x^{\ell_1}(V_{n-m_2})$ and
$V_{n-m_2}\cap\ker x^{\ell_1}=\ker x^{\ell_1}$. In view of
(\ref{L-1}), we conclude that
\[
x^{\ell_1}(V_{n-m_2})=V_{m_2},\quad \ker x^{\ell_1}\subset
V_{n-m_2},\quad \mbox{and so }\ V_{n-m_2}=(x^{\ell_1})^{-1}(V_{m_2})
\]
for all $(V_0,\ldots,V_n)\in\overline{\mathrm{Fl}_{x,d}(V)}$. \\
{\rm (b)} By assumption, $\lambda(x)$ (which is the shape of
$d=d_1+d_2$) is obtained by juxtaposing the shape of $d_1$ (formed
by $\ell_1$ columns) and the shape of $d_2$. Lemma \ref{L-linalg}
then implies that $\lambda(x_2)$ coincides with the shape of $d_2$.

Let $(V_0,\ldots,V_n)\in\mathrm{Fl}_{x,d}(V)$. Fix
$i\in\{0,1,\ldots,m_2\}$ and let us compare the nilpotent
endomorphisms
\[
x|_{V_{n-i}/V_i}:V_{n-i}/V_i\to V_{n-i}/V_i \ \mbox{ and }\
x|_{x^{\ell_1}(V_{n-i})/V_i}:x^{\ell_1}(V_{n-i})/V_i\to
x^{\ell_1}(V_{n-i})/V_i
\]
induced by $x$. On the one hand, by definition of
$\mathrm{Fl}_{x,d}(V)$, the Jordan form $\lambda(x|_{V_{n-i}/V_i})$
coincides with the shape $\lambda^{\lfloor\frac{n}{2}\rfloor-i}(d)$
of the subtableau of $d$ formed by the numbers $\leq \lfloor
\frac{n}{2}\rfloor-i$. Since $d=d_1+d_2$, the diagram
$\lambda^{\lfloor\frac{n}{2}\rfloor-i}(d)$ is obtained by
juxtaposing the shape of $d_1$ (consisting of $\ell_1$ columns) and
the shape $\lambda^{m_2-i}(d_2)$ of the subtableau of $d_2$ formed
by the numbers $\leq m_2-i$. On the other hand, we see that
\[x^{\ell_1}(V_{n-i})/V_i=\mathrm{Im}\,(x|_{V_{n-i}/V_i})^{\ell_1},\]
hence $x|_{x^{\ell_1}(V_{n-i})/V_i}$ is the restriction of
$x|_{V_{n-i}/V_i}$ to $\mathrm{Im}\,(x|_{V_{n-i}/V_i})^{\ell_1}$,
which implies (by Lemma \ref{L-linalg}) that the Jordan form
$\lambda(x|_{x^{\ell_1}(V_{n-i})/V_i})$ is the Young diagram
obtained from $\lambda(x|_{V_{n-i}/V_i})$ by removing the first
$\ell_1$ columns. Altogether, this yields
\[
\lambda(x|_{x^{\ell_1}(V_{n-i})/V_i})=\lambda^{m_2-i}(d_2)\ \mbox{
for all $i\in\{0,1,\ldots,m_2\}$}.
\]
Therefore, we have shown:
\[
(V_0,\ldots,V_n)\in\mathrm{Fl}_{x,d}(V)\quad\Rightarrow\quad
\varphi((V_0,\ldots,V_n))\in\mathrm{Fl}_{x_2,d_2}(\mathrm{Im}\,x^{\ell_1}).
\]
{\rm (c)} It follows from Lemma \ref{L-linalg}, applied to
$M=V_{m_2}$, that $\lambda(x_1)$ coincides with the diagram formed
by the first $\ell_1$ columns of $\lambda(x)$, i.e., the shape of
$d_1$.

Let $(W_0,\ldots,W_{n_1})\in\mathrm{Fl}_{x_1,d_1}(W(F))$, and let us
consider the element
\[\tilde{F}:=(V_0,\ldots,V_{m_2},\pi^{-1}(W_1),\ldots,\pi^{-1}(W_{n_1}),(x^{\ell_1})^{-1}(V_{m_2+1}),\ldots,(x^{\ell_1})^{-1}(V_{n_2})).\]
First note that $\dim \pi^{-1}(W_i)=\dim W_i+\dim V_{m_2}=m_2+i$ for
all $i\in\{1,\ldots,n_1\}$, while the rank formula yields
\[\dim
(x^{\ell_1})^{-1}(V_{m_2+i})=\dim V_{m_2+i}+\dim\ker
x^{\ell_1}=n_1+m_2+i\] for all $i\in\{1,\ldots,m_2\}$. In addition,
for every $i\in\{1,\ldots,n_1\}$, the $x_1$-stability of $W_i$
guarantees that $\pi^{-1}(W_i)$ is $x$-stable, while for every
$i\in\{1,\ldots,m_2\}$, the $x$-stability of $V_{m_2+i}$ implies
that $(x^{\ell_1})^{-1}(V_{m_2+i})$ is $x$-stable. Thereby
\[\tilde{F}\in\mathrm{Fl}_x(V).\]
For every $i\in\{0,\ldots,\lfloor\frac{n_1}{2}\rfloor\}$, there is a
natural isomorphism
$\pi^{-1}(W_{n_1-i})/\pi^{-1}(W_i)\stackrel{\sim}{\to}
W_{n_1-i}/W_i$, and the nilpotent maps
$x|_{\pi^{-1}(W_{n_1-i})/\pi^{-1}(W_i)}$ and
${x_1}|_{W_{n_1-i}/W_i}$ respectively induced by $x$ and $x_1$ are
conjugate under this linear isomorphism, hence
\begin{equation}
\label{shape1}
\lambda(x|_{\pi^{-1}(W_{n_1-i})/\pi^{-1}(W_i)})=\lambda({x_1}|_{W_{n_1-i}/W_i})=\lambda^{\lfloor
\frac{n_1}{2}\rfloor-i}(d_1)=\lambda^{\lfloor
\frac{n}{2}\rfloor-(m_2+i)}(d)
\end{equation}
(the last equality in (\ref{shape1}) follows from the assumption
that $d=d_1+d_2$). Next, let $i\in\{0,\ldots,m_2\}$, and let us
consider the nilpotent endomorphism
\[
\tilde{x}_i:=x|_{(x^{\ell_1})^{-1}(V_{n_2-i})/V_i}:(x^{\ell_1})^{-1}(V_{n_2-i})/V_i\to
(x^{\ell_1})^{-1}(V_{n_2-i})/V_i
\]
induced by $x$. We have
\[
\mathrm{Im}\,\tilde{x}_i^{\ell_1}=V_{n_2-i}/V_i \quad\mbox{and}\quad
\ker\tilde{x}_i^{\ell_1}=(x^{\ell_1})^{-1}(V_i)/V_i.
\]
The fact that $(V_0,\ldots,V_{n_2})\in
\mathrm{Fl}_{x_2,d_2}(\mathrm{Im}\,x^{\ell_1})$ and Lemma
\ref{L-linalg} (applied with $M=V_i$) respectively imply
\[
\lambda({\tilde{x}_i}|_{\mathrm{Im}\,\tilde{x}_i^{\ell_1}})=\lambda({x_2}|_{V_{n_2-i}/V_i})=\lambda^{m_2-i}(d_2)
\]
and
\[
\lambda({\tilde{x}_i}|_{\ker\tilde{x}_i^{\ell_1}})=\lambda(x|_{(x^{\ell_1})^{-1}(V_i)/V_i})=\lambda^1,
\]
where $\lambda^1$ stands for the diagram formed by the first
$\ell_1$ columns of $\lambda(x)$, that is, the shape of $d_1$ (since
$d=d_1+d_2$). Invoking again Lemma \ref{L-linalg}, we deduce that
the diagram $\lambda(\tilde{x}_i)$ is obtained by juxtaposing the
diagrams $\lambda^1$ and $\lambda^{m_2-i}(d_2)$. By definition of
the concatenation $d=d_1+d_2$, this yields
\begin{equation}
\label{shape2}
\lambda(\tilde{x}_i)=\lambda^{\lfloor\frac{n}{2}\rfloor-i}(d)\
\mbox{\ for all $i\in\{0,\ldots,m_2\}$.}
\end{equation}
Combining (\ref{shape1}) and (\ref{shape2}), we conclude that
$\tilde{F}\in\mathrm{Fl}_{x,d}(V)$.
\end{proof}

\begin{proof}[Proof of Proposition \ref{P12}]
Clearly, it is sufficient to show part {\rm (b)} of the statement.
Let $x_2=x|_{\mathrm{Im}\,x^{\ell_1}}$. Lemma \ref{L2}\,{\rm (b)}
and Example \ref{E1}\,{\rm (a)}--{\rm (b)} imply that
\[ \mathrm{Im}\,x^{\ell_1+1}=\mathrm{Im}\,x_2\subset V_{m_2}\subset\ker x_2=\ker x\cap\mathrm{Im}\,x^{\ell_1}\]
for all $F=(V_0,\ldots,V_n)\in \overline{\mathrm{Fl}_{x,d}(V)}$.
Note also that
$\dim \mathrm{Im}\,x^{\ell_1+1}=\mathrm{rank}\,x_2=k_2$.
This guarantees that the map $\Theta$ is well defined and algebraic.

Let $M=\ker x\cap\mathrm{Im}\,x^{\ell_1}/\mathrm{Im}\,x^{\ell_1+1}$. (Thus $\dim M=n_2-2k_2$.)
The group $\mathrm{SL}(M)$ embeds in $Z_G(x)$ in a natural way (see, e.g., \cite[Chapter 6]{Collingwood-McGovern}), so that $\Theta$ is in fact $\mathrm{SL}(M)$-equivariant.
In addition, the action of $\mathrm{SL}(M)$ on $\mathrm{Gr}_{m_2-k_2}(M)$ is transitive
and, by Bruhat's lemma, for every $L\in\mathrm{Gr}_{m_2-k_2}(M)$, there is
a closed subset $U=U(L)\subset\mathrm{SL}(M)$ (in fact, a unipotent subgroup)
such that the map
\[U\to\mathrm{Gr}_{m_2-k_2}(M),\ u\mapsto u(L)\]
is an open immersion.
Then, the map
\[U\times \Theta^{-1}(L)\to\Theta^{-1}(U\cdot L),\ (u,F)\mapsto u\cdot F\]
is a trivialization of $\Theta$ above $U\cdot L:=\{u(L):u\in U\}$.
We have shown that $\Theta$ is locally trivial.

It remains to determine the fiber $\Theta^{-1}(L)$ for a given
$L\in\mathrm{Gr}_{m_2-k_2}(M)$. In other words, $L$ can be viewed as
an $m_2$-dimensional space such that
\[\mathrm{Im}\,x^{\ell_1+1}\subset L\subset \ker x\cap\mathrm{Im}\,x^{\ell_1},\]
and $\hat{L}:=(x^{\ell_1})^{-1}(L)$ has dimension $\dim L+\dim\ker
x^{\ell_1}=n-m_2$. The set
\[X_L:=\{F=(V_0,\ldots,V_n)\in\mathrm{Fl}(V):V_{m_2}=L\ \mbox{and}\ V_{n-m_2}=\hat{L}\}\]
is a closed subvariety of $\mathrm{Fl}(V)$ and, by Lemma
\ref{L2}\,{\rm (a)}, we have
\begin{equation}
\label{fiber-1} \Theta^{-1}(L)=\overline{\mathrm{Fl}_{x,d}(V)}\cap
X_L.
\end{equation}
Note that there is a natural isomorphism
\[\Xi:\mathrm{Fl}(L)\times\mathrm{Fl}(\hat{L}/L)\times \mathrm{Fl}(V/\hat{L})\stackrel{\sim}{\to} X_L.\]
The nilpotent map $x^{\ell_1}$ induces a linear isomorphism
$V/\hat{L}\stackrel{\sim}{\to}\mathrm{Im}\,x^{\ell_1}/L$, hence
every element of $\mathrm{Fl}(V/\hat{L})$ is of the form
$((x^{\ell_1})^{-1}(V_{m_2}),\ldots,(x^{\ell_1})^{-1}(V_{n_2}))$ for
some
$(V_{m_2},\ldots,V_{n_2})\in\mathrm{Fl}(\mathrm{Im}\,x^{\ell_1}/L)$.
This observation, combined with Lemma \ref{L2}\,{\rm (c)}, the
description of
$\overline{\mathrm{Fl}_{x_2,d_2}(\mathrm{Im}\,x^{\ell_1})}$ (with
$d_2=d_{n_2,k_2}$) given in (\ref{5}), and (\ref{fiber-1}), yields
the inclusion
\begin{equation}
\label{fiber-2}
\Xi\big(\mathrm{Fl}(L)\times\overline{\mathrm{Fl}_{x_1,d_1}(\hat{L}/L)}\times
\mathrm{Fl}(V/\hat{L})\big)\subset\Theta^{-1}(L),
\end{equation}
where $x_1$ is the nilpotent map induced by $x$ on the subquotient
$\hat{L}/L=(x^{\ell_1})^{-1}(L)/L$. On the other hand, since
$\overline{\mathrm{Fl}_{x,d}(V)}$ and $\mathrm{Grass}_{m_2-k}(M)$
are irreducible, the fiber $\Theta^{-1}(L)$ must be irreducible, and
\[
\dim \Theta^{-1}(L) = \dim\overline{\mathrm{Fl}_{x,d}(V)}-\dim
\mathrm{Grass}_{m_2-k}(M).
\]
Invoking the equidimensionality of $\mathrm{Fl}_x(V)$, formula
(\ref{dimension-Flx}), and the fact that the Jordan form
$\lambda(x)$ is obtained by juxtaposing
$\lambda(x_1)$ and $\lambda(x_2)$ (see Lemma \ref{L2}), we have
\begin{eqnarray*}
\dim\overline{\mathrm{Fl}_{x,d}(V)} & = & \dim\mathrm{Fl}_{x}(V) =
\dim\mathrm{Fl}_{x_1}(\hat{L}/L)+\dim\mathrm{Fl}_{x_2}(\mathrm{Im}\,x^{\ell_1})
\\
 & = & \dim\overline{\mathrm{Fl}_{x_1,d_1}(\hat{L}/L)}+\binom{n_2-k_2}{2}+\binom{k_2}{2} \\
 & = & \dim\overline{\mathrm{Fl}_{x_1,d_1}(\hat{L}/L)}+2\binom{m_2}{2}+(m_2-k_2)^2
\\
 & = & \dim\overline{\mathrm{Fl}_{x_1,d_1}(\hat{L}/L)}+\dim
 \mathrm{Fl}(L)\times\mathrm{Fl}(V/\hat{L})+\dim\mathrm{Grass}_{m_2-k_2}(M).
\end{eqnarray*}
Whence
\[\dim \Theta^{-1}(L) = \dim \mathrm{Fl}(L)\times\overline{\mathrm{Fl}_{x_1,d_1}(\hat{L}/L)}\times
\mathrm{Fl}(V/\hat{L}).\] Therefore, equality holds in
(\ref{fiber-2}). The proof is complete.
\end{proof}

\subsection{Proof of Theorem \ref{T1}}

\label{section-5.3}

Let $x\in\mathfrak{g}(V,\omega)$ be a nilpotent element, i.e., a
nilpotent endomorphism of $V$ which is skew adjoint with respect to
the form $\omega$. Let $\lambda(x)\vdash n$ be the Jordan form of
$x$, seen as a Young diagram, and let $d=d_x^\omega$ be the
$\omega$-admissible domino tableau of shape $\lambda(x)$ defined in
(\ref{7}) (if $\omega$ is orthogonal) or (\ref{8}) (if $\omega$ is
symplectic). An immediate induction argument based on Example
\ref{E1}\,{\rm (a)}--{\rm (b)}, Proposition \ref{P12}, and the
construction of $d_x^\omega$ made in (\ref{7})--(\ref{8}) shows that
the irreducible component
$\overline{\mathrm{Fl}_{x,d_x^\omega}(V)}\subset\mathrm{Fl}_x(V)$
(see Section \ref{section-4.2}) is smooth. Combining this fact with
Propositions \ref{P10}--\ref{P11} completes the proof of the theorem.

\section{An example}

\label{section-6}

In this section we let $G=\mathrm{Sp}_6(\mathbb{K})$ and consider a
nilpotent element $x\in\mathfrak{sp}_6(\mathbb{K})$ of Jordan form
$\lambda(x)=(2,2,1,1)$. There are exactly three $\omega$-admissible
domino tableaux of shape $\lambda(x)$ (here $\omega$ is a symplectic
form on the space $V=\mathbb{K}^6$), namely
\[d_1=
\mbox{\scriptsize $\young(11,22,3,3)$}\,,\qquad d_2=\mbox{\scriptsize $\young(13,13,2,2)$}\,,\qquad d_3=\mbox{\scriptsize $\young(12,12,3,3)$}\,.\]
These domino tableaux yield three irreducible components
$\overline{\mathrm{Fl}_{x,d_1}(V)}$, $\overline{\mathrm{Fl}_{x,d_2}(V)}$, and $\overline{\mathrm{Fl}_{x,d_3}(V)}$
of the (type A) Springer fiber
$\mathrm{Fl}_x(V)$ (see Section \ref{section-4.2}).
They also give rise to a decomposition of the (type C) Springer fiber
\[\mathcal{B}_x\cong \mathrm{Fl}_x(V,\omega)=\overline{\mathrm{Fl}_{x,d_1}(V,\omega)}\cup \overline{\mathrm{Fl}_{x,d_2}(V,\omega)}\cup \overline{\mathrm{Fl}_{x,d_3}(V,\omega)}\]
and we know that each subvariety
$\overline{\mathrm{Fl}_{x,d_i}(V,\omega)}$ ($i\in\{1,2,3\}$) is a
union of irreducible components of $\mathrm{Fl}_x(V,\omega)$ (see
Proposition \ref{P9}). Recall that $A_x:=Z_G(x)/Z_G(x)^0$ stands for
the component group of the stabilizer of $x$; for the nilpotent
element $x$ considered in this section, we have
$A_x\cong\mathbb{Z}/2\mathbb{Z}$.

\begin{proposition}
\label{P13}
\begin{itemize}
%\item[\rm (a)] $\overline{\mathrm{Fl}_{x,d_1}(V,\omega)}$ consists
%of a single, smooth irreducible component.
\item[\rm (a)] $\overline{\mathrm{Fl}_{x,d_1}(V,\omega)}$ is smooth and irreducible.
\item[\rm (b)] $\overline{\mathrm{Fl}_{x,d_2}(V,\omega)}$
consists of two irreducible components which form a single $A_x$-orbit $\mathcal{C}$.
These components are smooth and their intersection is empty. The orbital variety corresponding to $\mathcal{C}$ in the sense of Proposition \ref{P1} is smooth.
\item[\rm (c)] $\overline{\mathrm{Fl}_{x,d_3}(V,\omega)}$ consists
of two irreducible components which form a single $A_x$-orbit
$\mathcal{C}'$. These components are smooth, but their intersection
is nonempty. This implies that the (type A) component
$\overline{\mathrm{Fl}_{x,d_3}(V)}$ is singular. This also implies
that the orbital variety corresponding to $\mathcal{C}'$ in the sense
of Proposition \ref{P1} is singular.
\item[\rm (d)] Thus the Springer fiber $\mathcal{B}_x\cong\mathrm{Fl}_x(V,\omega)$ has exactly five irreducible components
and all of them are smooth. The nilpotent orbit
$\mathcal{O}_x=G\cdot x$ contains exactly three orbital varieties,
two of them are smooth and one is singular.
\end{itemize}
\end{proposition}

\begin{proof}
We have $\mathrm{Im}\,x\subset\ker x$, $\dim \mathrm{Im}\,x=2$, and
$\dim\ker x=4$. Since $x$ is skew adjoint, the equality
\begin{equation}
\label{Imker} \ker x=(\mathrm{Im}\,x)^\perp \end{equation} holds. We
define the following closed subsets of $\mathrm{Fl}(V,\omega)$:
\begin{eqnarray*}
Z_1 & = & \{(V_0,\ldots,V_6)\in\mathrm{Fl}(V,\omega):\mathrm{Im}\,x\subset V_3\subset \ker x\}, \\
Z_2 & = & \{(V_0,\ldots,V_6)\in\mathrm{Fl}(V,\omega):x(V_5)\subset
V_1\}, \\
Z_3 & = & \{(V_0,\ldots,V_6)\in\mathrm{Fl}(V,\omega):V_2\subset\ker
x,\ x(V_4)\subset V_2\}.
\end{eqnarray*}
Whenever $(V_0,\ldots,V_6)\in Z_1\cup Z_2\cup Z_3$, the subspaces
$V_1,V_2,V_3$ are clearly $x$-stable, and since $x$ is skew adjoint,
$V_4(=V_2^\perp)$ and $V_5(=V_1^\perp)$ are $x$-stable as well. Thus
\[\mbox{$Z_1,Z_2,Z_3$ are subsets of $\mathrm{Fl}_x(V,\omega)$}.\]
Since the domino tableau $d_1$ coincides with the tableau $d_{6,2}$
of Example \ref{E1}\,{\rm (b)}, the inclusion
$\mathrm{Fl}_{x,d_1}(V,\omega)\subset Z_1$ follows from relation
(\ref{5}). The inclusion $\mathrm{Fl}_{x,d_2}(V,\omega)\subset Z_2$
follows from the definition of $\mathrm{Fl}_{x,d_2}(V,\omega)$ (note
that the condition $x(V_5)\subset V_1$ is equivalent to saying that
the nilpotent endomorphism $x|_{V_5/V_1}$ induced by $x$ has Jordan
form $(1,1,1,1)$). Finally, for every
$(V_0,\ldots,V_6)\in\mathrm{Fl}_{x,d_3}(V,\omega)$, we must have
$x(V_4)\subset V_2$ (since $x|_{V_4/V_2}$ has Jordan form $(1,1)$)
and $V_2\subset \ker x$ (since $x(V_2)\subset V_1\cap
\mathrm{Im}\,x=0$ where the last equality follows from the fact that
$\mathrm{rank}\,x|_{V_5/V_1}=\mathrm{rank}\,x=2$); whence
$\mathrm{Fl}_{x,d_3}(V,\omega)\subset Z_3$. Altogether we get the
inclusions \begin{equation} \label{inclusion}
\mathrm{Fl}_{x,d_i}(V,\omega)\subset Z_i\ \mbox{ for all
$i\in\{1,2,3\}$}\end{equation} and thereby (invoking
(\ref{partition}))
\begin{equation}
\label{union}
\mathrm{Fl}_x(V,\omega)=Z_1\cup Z_2\cup Z_3.
\end{equation}

For describing the structure of $Z_1$, $Z_2$, and $Z_3$, we need more
notation. By (\ref{Imker}), the subspace $\mathrm{Im}\,x$ is endowed
with a well-defined orthogonal form $\chi$ given by
\[\chi(x(v_1),x(v_2))=\omega(v_1,x(v_2))\quad\mbox{whenever $v_1,v_2\in V$.}\]
Since the space $\mathrm{Im}\,x$ is $2$-dimensional, it contains
exactly two $\chi$-isotropic lines $L_1$ and $L_2$. In addition
\begin{equation}
\label{Li} L_i^\perp=x^{-1}(L_i)\quad\mbox{for
$i\in\{1,2\}$}.\end{equation} It can also be seen that
\begin{equation}
\label{iota'} \mbox{there is an involutive element $h\in Z_G(x)$
such that $h(L_1)=L_2$}
\end{equation}
(more generally, the orthogonal group
$\mathrm{O}(\mathrm{Im}\,x,\chi)$ embeds in $Z_G(x)$; see, e.g.,
\cite[Chapter 6]{Collingwood-McGovern}). For $i\in\{1,2\}$, let
\begin{eqnarray*}
Z_{2,i} & = & \{(V_0,\ldots,V_6)\in\mathrm{Fl}(V,\omega):V_1=L_i\}, \\
Z_{3,i} & = & \{(V_0,\ldots,V_6)\in\mathrm{Fl}(V,\omega):L_i\subset
V_2\subset\ker x\}.
\end{eqnarray*}
Using (\ref{Li}), it is easy to see that
\[
Z_2=Z_{2,1}\cup
Z_{2,2}\quad\mbox{and}\quad
Z_3=Z_{3,1}\cup
Z_{3,2},
\]
hence (by (\ref{union}))
\[
\mathrm{Fl}_x(V,\omega)=Z_1\cup Z_{2,1}\cup Z_{2,2}\cup Z_{3,1}\cup Z_{3,2}.
\]

The map $Z_1\to \mathbb{P}(\ker
x/\mathrm{Im}\,x)\cong\mathbb{P}^1(\mathbb{K})$,
$(V_0,\ldots,V_6)\mapsto V_3/\mathrm{Im}\,x$ is a locally trivial
fiber bundle of fiber isomorphic to $\mathrm{Fl}(\mathbb{K}^3)$. For
$i\in\{1,2\}$, the subvariety $Z_{2,i}$ is isomorphic to the flag
variety $\mathrm{Fl}(L_i^\perp/L_i,\omega)$ in a natural way. In the
same way, for $i\in\{1,2\}$, the map $Z_{3,i}\to\mathbb{P}(\ker
x/L_i)\cong\mathbb{P}^2(\mathbb{K})$, $(V_0,\ldots,V_6)\mapsto
V_2/L_i$ is a locally trivial fiber bundle of typical fiber
$\mathbb{P}(V_2)\times\mathbb{P}(V_2^\perp/V_2)\cong
\mathbb{P}^1(\mathbb{K})\times \mathbb{P}^1(\mathbb{K})$. Altogether
we have shown that $Z_1,Z_{2,i},Z_{3,i}$ ($i\in\{1,2\}$) are smooth,
irreducible, $4$-dimensional subvarieties of
$\mathrm{Fl}_x(V,\omega)$. We conclude that
\begin{eqnarray*}
 & \mbox{$Z_1,Z_{2,1},Z_{2,2},Z_{3,1},Z_{3,2}$ are the irreducible
components of $\mathrm{Fl}_x(V,\omega)$,} \\
 & \mbox{all of them are smooth.}
\end{eqnarray*}

Since $\overline{\mathrm{Fl}_{x,d_i}(V,\omega)}$ (for $i\in\{1,2,3\}$)
are unions of irreducible components of $\mathrm{Fl}_x(V,\omega)$, the inclusions in (\ref{inclusion}) yield in fact
\begin{equation}
\label{equalities}
\overline{\mathrm{Fl}_{x,d_1}(V,\omega)}=Z_1,\quad
\overline{\mathrm{Fl}_{x,d_2}(V,\omega)}=Z_{2,1}\cup Z_{2,2},\quad
\overline{\mathrm{Fl}_{x,d_3}(V,\omega)}=Z_{3,1}\cup Z_{3,2}.
\end{equation}
The first equality in (\ref{equalities}) implies part {\rm (a)} of
the proposition. The second equality, combined with the observation
that $h(Z_{2,1})=Z_{2,2}$ (with $h$ in (\ref{iota'})) and
$Z_{2,1}\cap Z_{2,2}=\emptyset$, yields part {\rm (b)} of the
proposition. Finally, part {\rm (c)} of the statement is implied by
the last equality in (\ref{equalities}), combined with the
observation that $h(Z_{3,1})=Z_{3,2}$ and
\begin{equation}
\label{intersection}
Z_{3,1}\cap
Z_{3,2}=\{(V_0,\ldots,V_6)\in\mathrm{Fl}(V,\omega):V_2=\mathrm{Im}\,x\}\not=\emptyset;
\end{equation}
note that $Z_{3,1}$ and $Z_{3,2}$ are also irreducible components of
the fixed point set $(\overline{\mathrm{Fl}_{x,d_3}(V)})^\sigma$
(see the proof of Proposition \ref{P10}), so that
(\ref{intersection}) implies that
$(\overline{\mathrm{Fl}_{x,d_3}(V)})^\sigma$ and thus
$\overline{\mathrm{Fl}_{x,d_3}(V)}$ (see Proposition \ref{P3}) are
singular.
\end{proof}

\begin{remark}
{\rm (a)}
The number of irreducible components contained in
each subvariety $\overline{\mathrm{Fl}_{x,d_i}(V,\omega)}$ ($i\in\{1,2,3\}$)
and the action of $A_x$ on the set of components
can also be deduced from the results in \cite{van-Leeuwen}.
Moreover, the type A components $\overline{\mathrm{Fl}_{x,d_1}(V)}$ and $\overline{\mathrm{Fl}_{x,d_2}(V)}$ are smooth
(by Example \ref{E1} and Proposition \ref{P12}\,{\rm (a)}), which implies that the varieties
$\overline{\mathrm{Fl}_{x,d_1}(V,\omega)}$ and $\overline{\mathrm{Fl}_{x,d_2}(V,\omega)}$ are smooth (see the proof of Proposition \ref{P10}).
Proposition \ref{P13}\,{\rm (a)}--{\rm (b)} also follows from these observations.
\\
{\rm (b)} The singularity of the type A component $\overline{\mathrm{Fl}_{x,d_3}(V)}$, stated in Proposition \ref{P13}\,{\rm (c)},
is well known: it is already stated and shown in \cite{Spaltenstein} and \cite{Vargas}.
The novelty lies in our argument for showing the singularity of this component, which relies on the topology of type C components.
\end{remark}

\section*{Acknowledgement}

This work was conceived during a {\it Research in Pairs} stay at the Mathematisches Forschungsinstitut Oberwolfach in January 2015. We thank the Forschungsinstitut and its staff for the support and the
delightful work conditions.

\end{document}